\documentclass[11pt]{amsart}
\usepackage{verbatim, latexsym, amssymb, amsmath}
\usepackage{epsfig}
\usepackage{hyperref}
\usepackage{float}
\usepackage{graphicx}
\usepackage{subfig}
\usepackage{setspace}
\usepackage{color}
\usepackage{}

\newtheorem{thm}{Theorem}[section]
\newtheorem{lemm}[thm]{Lemma}
\newtheorem{cor}[thm]{Corollary}
\newtheorem{defi}[thm]{Definition}
\newtheorem{prop}[thm]{Proposition}
\newtheorem*{thmA}{Theorem A}
\newtheorem*{thmB}{Theorem B}
\newtheorem*{thmC}{Theorem C}
\newtheorem*{thmD}{Theorem D}
\newtheorem*{thmE}{Theorem E}
\newtheorem*{thmF}{Theorem F}

\theoremstyle{remark}
\newtheorem{rmk}[thm]{Remark}

\theoremstyle{definition}

\newtheorem{exam}{Example}

\numberwithin{equation}{subsection}

\title[The width of embedded circles]{The width of embedded circles}
\author{Lucas Ambrozio, Rafael Montezuma and Roney Santos}
\address{L. Ambrozio: IMPA \\ Rio de Janeiro RJ 22460-320
Brazil}
\email{l.ambrozio@impa.br}

\address{R. Montezuma: Departamento de Matemática, UFC \\ Campus do Pici \\ Fortaleza, CE 60455-760 Brazil}
\email{montezuma@mat.ufc.br}

\address{R. Santos: Universidade de S\~ao Paulo \\ Departamento de Matem\'atica \\ Rua do Matão, São Paulo, SP 05508-900 Brazil}
\email{roneypsantos@ime.usp.br}

\thanks{L.A. was supported by CNPq - Conselho Nacional de Desenvolvimento Cient\'ifico e Tecnol\'ogico (309908/2021-3 - Bolsa PQ) and by FAPERJ - Funda\c{c}\~ao Carlos Chagas Filho de Amparo \`a Pesquisa do Estado do Rio de Janeiro (grant SEI-260003/000534/2023 - BOLSA E-26/200.175/2023 and grant SEI-260003/001527/2023 - APQ1 E-26/210.319/2023)}
\thanks{R.M. was supported by Instituto Serrapilheira grant ``New perspectives of the min-max theory for the area functional" and by CNPq - Conselho Nacional de Desenvolvimento Cient\'ifico e Tecnol\'ogico (311028/2020.9 - Bolsa PQ)}
\thanks{R.S. was supported by Instituto Serrapilheira grant ``New perspectives of the min-max theory for the area functional".}

\begin{document}

\begin{abstract}
	We develop a Morse-Lusternik-Schnirelmann theory for the distance between two points of a smoothly embedded circle in a complete Riemannian manifold. This theory suggests very naturally a definition of width that generalises the classical definition of the width of plane curves. Pairs of points of the circle realising the width bound one or more minimising geodesics that intersect the curve in special configurations. When the circle bounds a totally convex disc, we classify the possible configurations under a further geometric condition. We also investigate properties and characterisations of curves that can be regarded as the Riemannian analogues of plane curves of constant width.
\end{abstract}

\maketitle

\section{Introduction}

\subsection{The width of convex planar curves} \label{introplanar} Let $\Gamma$ be an embedded circle in the Euclidean plane. Given any unit vector $\theta\in \mathbb{R}^2$, the \textit{width of $\Gamma$ in the direction $\theta$} is the smallest distance $w(\theta)$ between two parallel lines that are both orthogonal to $\theta$ and contain $\Gamma$ between them. The maximum of $w(\theta)$ is nothing but the (extrinsic) \textit{diameter} of $\Gamma$, while the minimum is called simply the \textit{width} of $\Gamma$, and denoted $w(\Gamma)$. For instance, if $\Gamma$ is an ellipse, $w(\Gamma)$ is the length of the minor axis, while the length of the major axis is equal to the diameter of $\Gamma$.

	If the embedding is of class $C^1$, it is straightforward to show that $w(\Gamma)$ is equal to the length of a line segment that meets $\Gamma$ at its extremities, in right angles. For instance, if such $\Gamma$ is also strictly convex, then for every $\theta$ such that $w(\theta)=w(\Gamma)$, the two parallel lines orthogonal to $\theta$ and containing $\Gamma$ between them that are closest to each other touch $\Gamma$ tangentially at exactly two points, which are the extremities of a line segment of this sort. In this case, this segment lies in the convex region bounded by $\Gamma$, and intersects $\Gamma$ only at its extremities. For non-convex curves, however, this might not be the case. For an example of non-convex plane region, bounded by a smooth embedded curve, that in fact contains no segment orthogonal to the boundary curve, see Figure 1 in \cite{Bos}. 
	
	The function $\theta\in S^1 \mapsto w(\theta)\in (0,+\infty)$ contains interesting information about of the geometry of $\Gamma$. For instance, the \textit{Cauchy-Crofton formula} for convex curves in the Euclidean plane \cite{San} computes the length of such rectifiable curves $\Gamma$ from the information about the width of $\Gamma$ in almost every direction. More precisely, if $\Gamma$ is convex, we have
\begin{equation*}
	\int_{S^1} w(\theta) d\theta = 2 L(\Gamma).
\end{equation*}
An immediate and interesting consequence of this formula is the inequality
\begin{equation*}
	\frac{w(\Gamma)}{L(\Gamma)} \leq \frac{1}{\pi},
\end{equation*}
which is an equality if and only if $\Gamma$ has the same width in all directions, that is, $\Gamma$ is a \textit{curve of constant width}. These curves are also characterised as those for which the width $w(\Gamma)$ equals the diameter of $\Gamma$. \\

	More generally, consider a smoothly embedded circle $\Gamma$ in a complete Riemannian surface. What is the right generalisation of the ``width" for such objects? Can it also be realised as the length of a certain minimising geodesic intersecting the curve orthogonally? How does this notion of ``width" compare to the diameter and the length of the curve? What should be the analogues of plane curves of constant width in this more general setup?
	
	We investigate these and other related questions, adopting the perspective of a min-max variational theory for a non-local, geometric functional.

\subsection{Lusternik-Schnirelmann theory} \label{introheuristics} The set $\mathcal{P}$ of subsets of the circle $S^1$ with at most two elements is a compact manifold that can be identified with the set of ordered pairs $(p,q)\in S^1\times S^1$ modulo the equivalence relation that identifies $(p,q)$ and $(q,p)$. In particular, $\mathcal{P}$ is homeo\-mor\-phic to a M\"obius band, whose boundary points are precisely the singletons $\{p\}\subset S^1$. The elements of $\mathcal{P}$ will be denoted by $\{p,q\}$ (where, of course, we may have $p=q$).

	If $\Gamma$ is a smoothly embedded circle in some complete Riemannian manifold $(M^n,g)$, we define a functional $\mathcal{D} : \mathcal{P} \rightarrow \mathbb{R}$ by setting 
\begin{equation*}
	\mathcal{D}(\{p,q\}) = d(p,q) \quad \text{for all} \quad p,q\in \Gamma,
\end{equation*}
where $d : M\times M \rightarrow [0,+\infty)$ is the Riemannian distance of $(M^n,g)$.

	Clearly, $\mathcal{D}$ is a non-negative continuous functional. While its absolute minimum is zero and attained precisely at the subsets with exactly one element, its maximum is just the \textit{diameter} of $\Gamma$ in $(M^n,g)$, that is, the maximum distance in $(M^n,g)$ between pairs of points of $\Gamma$.
	
	This particular bounded continuous function on $\mathcal{P}$ has more interesting properties, though. In fact, it passes down continuously to the quotient $\mathcal{P}_{*}$ of $\mathcal{P}$ by the relation that identifies all boundary points of $\mathcal{P}$ as a single point. By doing so, we can regard $\mathcal{D}$ as a continuous non-negative function on the real projective plane $\mathcal{P}_*$, which attains its minimum value zero and its maximum value equal to the diameter of $\Gamma$ in $(M^n,g)$.
	
	Inspired by the Morse and Lusternik-Schnirelmann theories about the number of critical points of smooth functions on a real projective plane, one could hope to detect another ``critical point" of $\mathcal{D}$ at some ``critical level" between the other two. The general notion of ``width" of the curve $\Gamma$ in $(M^n,g)$ that we will propose captures precisely this intuition.
	
	Before we continue, there are two issues of technical nature that are worth to be highlighted, because the geometric functional $\mathcal{D}$ has two interesting features from a variational perspective that pose new difficulties for the min-max method.
	
	The first feature is the non-local character of $\mathcal{D}$. By this we just mean that the geometry of $(M^n,g)$ must be known in order to compute the distance between pairs of points in $\Gamma$. In that regard, one can make an analogy between the ``critical value" that we are seeking to define and the first Steklov eigenvalue of a compact Riemannian surface with connected boundary. (\textit{Cf}. Section 4.1 of the survey article \cite{GirPol}).
	
	The second feature is the fact that the Riemannian distance function is not smooth in general. For instance, it is not smooth if there is a pair of points that is connected by two or more minimising geodesics. While in Euclidean space this does not happen, in other Riemannian manifolds this is a possibility.
	
	Nevertheless, the Riemannian distance function is well-behaved enough so to allow a distinction between ``critical" and ``regular" points of $\mathcal{D}$, and the development of a Morse theory. In the case of the distance function to a fixed point in a complete Riemannian manifold, the suitable notions were introduced by Grove and Shiohama \cite{GroShi}, and yielded important results in Riemannian Geometry (for instance, the Diameter Sphere Theorem \cite{GroShi} and Gromov's Betti Number Theorem \cite{Gro}). As it will become clearer from the discussion that follows, we will eventually prove that the ``critical value" of $\mathcal{D}$ predicted by the above informal discussion is indeed attained as the value of $\mathcal{D}$ at some ``critical point".

\subsection{Regular and critical points} \label{introcriticalpoints} Let $(M^n,g)$ be a complete Riemannian manifold, and let $\Gamma$ be a smoothly embedded circle in $(M^n,g)$.

	Given points $p$ and $q$ in $M$, let $\gamma:[0,a]\rightarrow M$ be a geodesic of $(M^n,g)$ joining $p$ and $q$. If $\gamma$ is not the trivial geodesic, we always assume it to be normalised, \textit{i.e.} $|\gamma'|=1$. A geodesic $\gamma$ is called \textit{minimising} if its length $L(\gamma)$ is equal to the Riemannian distance $d(p,q)$ between $p$ and $q$ in $(M^n,g)$.
	
	By the first variation formula of the length, if $c_t$ is a smooth variation of a geodesic $c_0=\gamma$ by smooth curves in $M$ joining points of $\Gamma$, then
\begin{equation}\label{eqfirstvariation}
	\frac{d}{dt}_{|_{t=0}}L(c_t) = \langle v_2,\gamma'(a)\rangle - \langle v_1,\gamma'(0)\rangle = \langle v_2,(\gamma'(a))^T\rangle - \langle v_1,(\gamma'(0))^T\rangle,
\end{equation}
where $v_1\in T_p\Gamma$, $v_2\in T_q\Gamma$ are the boundary values of the variational vector field $d c_t/dt_{|{t=0}}$ along $\gamma$, and $(\gamma(0)')^T$, $(\gamma(a)')^T$ denote orthogonal projections onto $T\Gamma$. (We follow standard conventions and use $\langle-,-\rangle$ for $g$ sometimes).

	As before, denote by $\mathcal{P}$ the set of subsets $\{x,y\}\subset \Gamma$ with at most two elements, and let $\mathcal{D}:\mathcal{P}\rightarrow [0,+\infty)$ be the functional that assigns to each $\{x,y\}\in \mathcal{P}$ the Riemannian distance between $x$ and $y$ in $(M^n,g)$, $\mathcal{D}(\{x,y\})=d(x,y)$. Motivated by the above remarks, we adopt the definition below. 
\begin{defi} \label{defcriticalpoints}
The set $\{p,q\}\subset \Gamma$ is called a regular point of $\mathcal{D}$ if there exists a vector $(v_1,v_2)\in T_p\Gamma\times T_q\Gamma$ such that, for every minimising geodesic $\gamma$ joining $p=\gamma(0)$ to $q=\gamma(a)$, we have
	\begin{equation*}
		\langle v_2,\gamma'(a)\rangle - \langle v_1,\gamma'(0)\rangle < 0.
	\end{equation*}
	The set $\{p,q\}\subset \Gamma$ is called a critical point  of $\mathcal{D}$ if it is not a regular point. 
\end{defi}
	
	Thus, if $\{p,q\}\subset \Gamma$ is a regular value, there exists a vector field $V$ on $M$, that is tangent to $\Gamma$, and whose flow $\phi_t$ has the following property: For \textit{every} minimising geodesic $\gamma$ in $(M^n,g)$ joining $p$ and $q$, the curves $\phi_t(\gamma)$ for small $t>0$ are curves with extremities in $\Gamma$ such that $L(\phi_t(\gamma))<L(\gamma)=d(p,q)$. In particular, the extremities of the curves $\phi_t(\gamma)$, for $t>0$ small enough, are at distance strictly smaller than $d(p,q)$.
	
	From Definition \ref{defcriticalpoints}, it is clear that all singletons $\{p\}\subset \Gamma$ are trivially critical points. More interestingly, if $p$ and $q$ are the extremities of a \textit{free boundary} minimising geodesic, \textit{i.e.} a minimising geodesic that is orthogonal to $\Gamma$ at both extremities, then $\{p,q\}\subset \Gamma$ is a critical point of $\mathcal{D}$. 
	
	There are, however, many other types of critical points. For instance, if $p$, $q\in \Gamma$ are points that are joined by two minimising geodesics $\gamma_1, \gamma_2:[0,a]\rightarrow M$ that are \textit{simultaneously stationary} in the sense that there exists a constant $c>0$ such that
\begin{equation*}
	(\gamma'_1(0))^T=-c(\gamma'_2(0))^T \quad \text{and} \quad (\gamma'_1(a))^{T}=-c(\gamma'_2(a))^{T},
\end{equation*}
then $\{p,q\}$ is a critical point of $\mathcal{D}$ as well. There are smoothly embedded circles in Riemannian manifolds such that all non-trivial critical points of $\mathcal{D}$ bound a pair of simultaneously stationary minimising geodesics, while none of them are the extremities of a free boundary minimising geodesic, see Example \ref{examellipsoids} in Section \ref{sectionexamples}.

	Free boundary minimising geodesics and pairs of simultaneously stationary minimising geodesics will play an important role when $\Gamma$ is the boundary of a totally convex disc inside a Riemannian surface, see Section \ref{introdiscs}.
	
	Finally, we remark that when the distance function $d$ of $(M^n,g)$ restricted to pairs of different points is a smooth function, it is immediate to check that Definition \ref{defcriticalpoints} captures precisely the standard regular and critical points of the restriction of $d$ to pairs of different points of $\Gamma$. This is the case, for instance, of Cartan-Hadamard manifolds.

\subsection{The width of a curve}\label{introwidth} Let $\Gamma$ be a smoothly embedded circle in a complete Riemannian manifold $(M^n,g)$. Following the discussion in Section \ref{introheuristics}, we propose a definition of the width of $\Gamma$. First, we define certain closed loops in the projective plane $\mathcal{P}_*$ that we call \textit{sweepouts}.

\begin{defi}\label{defiweepout}
	A family of pairs of points $\{p_t,q_t\}\subset \Gamma$, $t\in [0,1]$, is called a \textit{sweepout} of $\Gamma$ when the following conditions hold:
\begin{itemize}
	\item[$i)$] $p_0=q_0$ and $p_1=q_1$.
	\item[$ii)$] $t\in [0,1] \mapsto p_t\in \Gamma$ and $t\in [0,1] \mapsto q_t\in \Gamma$ are continuous maps.
	\item[$iii)$] $p_t$ and $q_t$ bound a closed arc $C_t\subset \Gamma$ in such way that $t\in [0,1] \mapsto C_t$ is a continuous map with $C_0=\{p_0\}=\{q_0\}$ and $C_1=\Gamma$.
\end{itemize}
\end{defi}
	
	Here, the distance between arcs $C_1$, $C_2 \subset \Gamma$ is measured by the length of the symmetric difference $C_1\Delta C_2 = (C_1\setminus C_2)\cup (C_2\setminus C_1)$.
	
	Sweepouts can be regarded as homotopically non-trivial loops in the projective plane $\mathcal{P}_*$. Indeed, this space is doubly covered by the set $\mathcal{O}_*$ formed by ordered pairs $(C_1,C_2)$ consisting of closed arcs of the circle $\Gamma$ that share only their extremities $\partial C_1 = \partial C_2=C_1\cap C_2$. The covering map is $(C_1,C_2)\in \mathcal{O}_* \mapsto \partial{C_1}=\partial C_2 \in \mathcal{P}_*$. Then, a continuous path $t\in [0,1]\mapsto \{p_t,q_t\}\subset \mathcal{P}_*$ is a sweepout of $\Gamma$ if and only if the continuous lift $t\in[0,1]\mapsto (C_t,\Gamma\setminus int(C_t))\in \mathcal{O}_*$, where $\partial C_t=\{p_t,q_t\}$ for $t\in (0,1)$, starts at $(\{p_0\},\Gamma)$ and finishes at $(\Gamma,\{p_1\})\neq (\{p_0\},\Gamma)$.
\begin{defi}\label{defwidth}
	Let $\Gamma$ be a smoothly embedded circle in a complete Riemannian manifold $(M^n,g)$. The width of $\Gamma$ is the number
	\begin{equation*}
		\mathcal{S}(\Gamma) = \inf_{\{p_t,q_t\}\in \mathcal{V}} \,\, \max_{t\in [0,1]} d(p_t,q_t)
	\end{equation*}
	where $\mathcal{V}$ is the set of sweepouts $\{p_t,q_t\}$ of $\Gamma$. 	
\end{defi}
	Since the subset of points that divides $\Gamma$ into two arcs of equal length form a compact subset of $\mathcal{P}$, and since every sweepout of $\Gamma$ contains points $\{p_t,q_t\}\subset \Gamma$ with this property because of condition $iii)$, we have
\begin{equation*}
	\mathcal{S}(\Gamma)>0.
\end{equation*}

	If $\Gamma$ is the boundary of a region $\Omega$ in a complete Riemannian surface $(M^2,g)$, and if all minimising geodesics joining points of $\Omega$ lie in $\Omega$, the distance between pairs of points in $\Gamma=\partial \Omega$ depends only on $(\Omega,g)$, and $\mathcal{S}(\partial \Omega)$ can be regarded as a Riemannian invariant of $(\Omega,g)$.

\subsection{The basic min-max theorem} We are now ready to formulate the most basic version of the min-max theorem for the functional $\mathcal{D}$.

\begin{thmA}
	Let $\Gamma$ be a smoothly embedded circle in a complete Riemannian manifold $(M^n,g)$. Then $\mathcal{S}(\Gamma)>0$ and there exists a critical point  $\{p,q\} \subset \Gamma$ of $\mathcal{D}$ such that
	\begin{equation*}
		\mathcal{D}(\{p,q\})=\mathcal{S}(\Gamma).
	\end{equation*}
\end{thmA}
	An immediate consequence of Definition \ref{defcriticalpoints} is that, if some non-trivial critical point $\{p,q\}\subset \Gamma$ bounds only one minimising geodesic, then this minimising geodesic meets $\Gamma$ orthogonally both at $p$ and at $q$. Thus, if all pairs of points of $(M^n,g)$ are joined by a unique minimising geodesic, Theorem A guarantees that $\Gamma$ is always cut orthogonally by some geodesic that is, at the same time, the curve of least length in $(M^n,g)$ joining its extremities in $\Gamma$. This is the case for Cartan-Hadamard manifolds, for example.
	
	On the other hand, there are examples of curves $\Gamma$ bounding convex regions of Riemannian surfaces such that $\mathcal{S}(\Gamma)$ is not attained as the length of a free boundary minimising geodesic, see Example \ref{examellipsoids} in Section \ref{sectionexamples}.
	
	The preceding remarks already points to the fact that Theorem A detects a different geometric feature of certain Riemannian surfaces with convex, connected boundary $(\Omega,g)$ than other notions of width arising from min-max procedures designed to produce free boundary geodesics on such surfaces. Before we elaborate on this point in Section \ref{introcomparison}, we discuss what else can be said about the minimising geodesics bounded by a non-trivial critical point of $\mathcal{D}$ in $\partial \Omega$, when $\Omega$ is a disc.

\subsection{The width of the boundary of discs} \label{introdiscs} 
	Let $\Omega$ be a totally convex disc with smooth boundary inside a complete Riemannian surface $(M^2,g)$. Recall that a subset of a Riemannian manifold is called \textit{totally convex} when every geodesic with extremities in the subset lies entirely in the subset. Thus, for most of our purposes, we may well forget about the ambient surface and work within $(\Omega,g)$. (A large class of examples of these objects is described in Example \ref{examconvex} in Section \ref{sectionexamples}).
	
	The boundary of a totally convex disc $(\Omega,g)$ is \textit{convex} in $\Omega$ in the sense that the second fundamental form $A$ of $\partial \Omega$, with respect to the outward pointing unit normal $N$, satisfies $A(X,X)=\langle\nabla_X N,X\rangle\geq 0$ for all vector fields $X$ that are tangent to $\partial \Omega$. (If $X$ has norm one, $A(X,X)$ is nothing but the \textit{geodesic curvature} of $\partial \Omega$ in $(\Omega,g)$). If the strict inequality holds for all non-zero tangent vectors to the boundary, we say that the boundary of $\Omega$ is \textit{strictly convex}.
	
	If the extremities of a geodesic $\gamma$ lie in $\partial \Omega$, we say that $\gamma$ is \textit{proper} whenever it intersects $\partial \Omega$ only at its extremities. By the convexity assumption on $\Omega$, a non-proper geodesic joining two boundary points must be itself part of $\partial \Omega$. In particular, if the boundary is strictly convex, then all geodesics joining boundary points are proper.
	
	Under these geometric conditions, critical points of $\mathcal{D}$ on $\partial \Omega$ and the minimising geodesics they bound enjoy further properties. For instance, a uniqueness property holds (Proposition \ref{propuniqu}), and local minima of $\mathcal{D}$ bound exactly one minimising geodesic (Proposition \ref{proplocalminima}). Moreover, Theorem A can be refined to yield much more precise information about critical points of $\mathcal{D}$ at the level $\mathcal{S}(\partial \Omega)$, at least under a natural geometric condition that we will introduce shortly.\\
\indent First, we need to recall the definition of the Morse index of a free boundary geodesic. Given a free boundary proper geodesic $\gamma : [0,a]\rightarrow \Omega$ and a vector field $X$ along $\gamma$ that is normal to $\gamma$ (and therefore tangent to $\Gamma$ at the extremities of $\gamma$), define
\begin{multline*}
	Q(X,X) =  \int_{0}^{a}  \left(|\nabla_{\gamma'} X|^2 - K|X|^2\right)dt \\ - A(X(\gamma(0)),X(\gamma(0)))- A(X(\gamma(a)),X(\gamma(a))),
\end{multline*}
where $K$ is the Gaussian curvature of $(\Omega,g)$.

	This quadratic expression in $X\in \Gamma(N\gamma)$ appears when one computes the second variation of the length of a smooth family of smooth curves $c_t$, with extremities in $\partial \Omega$, starting at $c_0=\gamma$ and with variational vector field $X$ along $\gamma$. In fact, in this case
\begin{equation}\label{eqsecondvariation}
	\frac{d^2}{dt^2}_{|_{t=0}} L(c_t) = Q(X,X). 
\end{equation}
	
	The geodesic $\gamma$ is called \textit{free boundary stable} if $Q(X,X)\geq 0$ for every $X\in \Gamma(N\gamma)$, and \textit{free boundary unstable} otherwise. The \textit{free boundary index} of the geodesic $\gamma$ is the index of the quadratic form $Q$. (If $\gamma$ is minimising, then the second variation formula \eqref{eqsecondvariation} implies $Q(X,X)\geq 0$ for every $X\in \Gamma(N\gamma)$ that vanishes at the extremities, but this condition does not imply free boundary stability).
	
	For instance, consider $X$ so that $\{\gamma'(s),X(\gamma(s))\}$ is an orthonormal basis for all $s$. Then $Q(X,X)<0$ if the Gaussian curvature of $(\Omega,g)$ is non-negative and its boundary is strictly convex, or if the Gaussian curvature is positive and the boundary is convex. In other words, under these geometric assumptions, no free boundary geodesic of $(\Omega,g)$ is free boundary stable.
	
	The geometric assumption
\begin{equation*}
	(\star) \textit{ no free boundary stable geodesic exists on } (\Omega,g)
\end{equation*}
allows a rather detailed understanding of the possible configurations of minimising geodesics bounded by the critical points whose existence is guaranteed by Theorem A. 
\begin{thmB}
	Let $(\Omega,g)$ be a totally convex disc with smooth boundary in a complete Riemannian surface. Assume $(\Omega,g)$ has property $(\star)$. Then every critical point $\{p,q\}\subset \partial \Omega$ of $\mathcal{D}$ with $p\neq q$ satisfies
	\begin{equation*}
		d(p,q)\geq \mathcal{S}(\partial \Omega).
	\end{equation*} 
	Moreover, let $\{p,q\}\subset \partial \Omega$ be a critical point with $d(p,q)= \mathcal{S}(\partial \Omega)$. Then:
	\begin{itemize}
		\item[$i)$] If there exists a free boundary minimising geodesic joining $p$ and $q$, then this geodesic has free boundary index one.
		\item[$ii)$] If two different minimising geodesics $\gamma_1, \gamma_2:[0,a]\rightarrow \Omega$ joining the points $p$ and $q$ satisfy 
		\begin{equation*}
			\langle (\gamma'_1(0))^T,(\gamma'_2(0))^T\rangle \leq 0 \quad \text{and} \quad \langle (\gamma'_1(a))^T,(\gamma'_2(a))^T\rangle \leq 0,
		\end{equation*}
		and none of them is free boundary, then $\gamma_1$ and $\gamma_2$ are simultaneously stationary.
	\end{itemize}
\end{thmB}
	
	This theorem and its proof are reminiscent of some results of Marques and Neves \cite{MarNev-Duke}. There are examples of Riemannian discs that satisfy the assumptions of Theorem B and provide examples of both behaviours described in $i)$ and $ii)$, see Example \ref{examellipsoids} in Section \ref{sectionexamples}.
	
	On the other hand, if the assumption $(\star)$ is dropped, there are examples where one finds critical points $\{x,y\}\subset \partial \Omega$ of $\mathcal{D}$ with $d(x,y)<\mathcal{S}(\partial \Omega)$, see Example \ref{examhalfdumbbell} in Section \ref{sectionexamples}.
	
	Plane convex regions with smooth boundary are totally convex regions of the plane, and among them the strictly convex ones satisfy assumption $(\star)$. Therefore, a consequence of Theorem B is that the geometric invariant $\mathcal{S}(\partial \Omega)$ of a strictly convex plane region $\Omega$ with smooth boundary is equal to the width of $\partial \Omega$ discussed in Section \ref{introplanar}. This gives some justification to our choice of terminology. Moreover, since plane smooth convex curves can be approximated by plane smooth strictly convex curves (\textit{e.g.} by flowing it a bit by the curve shortening flow), by a straightforward continuity argument we can actually conclude:

\begin{cor}\label{corcoincidencewidths}
	Let $\Omega$ be a compact, convex, regular domain of the Euclidean plane. Then $\mathcal{S}(\partial \Omega)=w(\partial \Omega)=\min_{\theta\in S^1}w(\theta)$.	 
\end{cor}

	Non-convex plane smooth curves might be such that $\mathcal{S}(\Gamma)<w(\Gamma)$, though. For instance, a thin U-shapped curve $\Gamma$ has strictly larger widths $w(\theta)$ in all directions than the value of $\mathcal{S}(\Gamma)$.

	Another consequence of Theorem B is the following existence theorem, of independent interest:

\begin{cor}\label{corexistence}
	Let $(\Omega,g)$ be a totally convex disc with smooth boundary in a complete Riemannian surface. Assume $(\Omega,g)$ has property $(\star)$. Then, either there exists a free boundary proper minimising geodesic of free boundary index one, or there exists a pair of simultaneously stationary minimising geodesics with extremities in $\partial \Omega$.
\end{cor}

\subsection{Width, diameter and length} Let $\Gamma$ be a smoothly embedded circle in a complete Riemannian manifold. Clearly, the distance between two points of $\Gamma$ is at most the diameter of $\Gamma$, which by its turn is bounded by the length of the shortest arc of $\Gamma$ bounded by them. Thus, 
\begin{equation*}
	\mathcal{S}(\Gamma) \leq diam(\Gamma) \leq \frac{1}{2}L(\Gamma),
\end{equation*} 
Notice that $L(\Gamma)/2$ is nothing but the \textit{(intrinsic) diameter} of $(\Gamma,g_{|_{\Gamma}})$.

	In Section \ref{sectionexamples}, we describe examples where equality holds between the first and second numbers, or between the second and the third numbers, without all of them being equal, see Examples \ref{examellipsoids} and \ref{examellipsoids3}.
	
	It is interesting to characterise when any of the above inequalities is an equality. The most interesting case is perhaps the equality between the width and the diameter, because plane curves of constant width have this property.

\begin{thmC}
	Let $\Gamma$ be a smoothly embedded circle in a complete Riemannian manifold.
	\begin{itemize}
		\item[$a)$] If there exists a continuous map $\phi : \Gamma \rightarrow \Gamma$ such that $d(x,\phi(x))=diam(\Gamma)$, then $\mathcal{S}(\Gamma)=diam(\Gamma)$.
		\item[$b)$] If $\mathcal{S}(\Gamma)=diam(\Gamma)$, then for every $x\in \Gamma$ there exists $y\in \Gamma$ such that $d(x,y)=diam(\Gamma)$.
	\end{itemize}
	Assume, moreover, that $\Gamma$ is the boundary of a totally convex smoothly embedded disc. Then $\mathcal{S}(\Gamma)=diam(\Gamma)$ if and only if there exists a continuous map $\phi : \Gamma \rightarrow \Gamma$ such that $d(x,\phi(x))=diam(\Gamma)$.
\end{thmC}

\indent The equality between the extrinsic and intrinsic diameters of a curve $\Gamma$, on the other hand, is easily characterised. In fact, $diam(\Gamma)=L(\Gamma)/2$ if and only if $\Gamma$ is a closed geodesic formed by two minimising geodesics of the same length $L(\Gamma)/2$, see Lemma \ref{lemdiamlength}.

	Finally, we characterise the equality between the three geometric invariants above as follows:
\begin{thmD}
	Let $\Gamma$ be a smoothly embedded circle in a complete Riemannian manifold. The following assertions are equivalent:
	\begin{itemize}
		\item[$i)$] $\mathcal{S}(\Gamma)=L(\Gamma)/2$.
		\item[$ii)$] For every $x$, $y\in \Gamma$, the distance between $x$ and $y$ equals the length of the shortest arc of $\Gamma$ bounded by these two points.
	\end{itemize}
\end{thmD}
	
	In particular, under any of these conditions, $\Gamma$ is a geodesic such that any points $x$, $y\in \Gamma$ bounding arcs of $\Gamma$ of the same length lie at distance $d(x,y)=L(\Gamma)/2$. Moreover, the pair $\{x,y\}$ divides $\Gamma$ in two minimising geodesics. (See also Remark \ref{rmkjacobi}). 
	
	It is interesting to contrast the objects described in Theorem D to Riemannian fillings of the circle \cite{Gro1}. Recall that a Riemannian filling of the circle is a Riemannian surface $(M^2,g)$ with connected boundary such that
\begin{equation*}
	d_{(M,g)}(x,y)=d_{(\partial M,g_{|_{\partial M}})}(x,y) \quad \text{for all} \quad x,y\in \partial M.
\end{equation*}
In Section \ref{sectionexamples}, we describe several examples of fillings, see Example \ref{examfilling}.

	In our view, the comparison between the width and the boundary length of a curve is, to a certain extent, analogous to the comparison between the first Steklov eigenvalue and the boundary length of a compact surface \cite{GirPol}, or even to the isosistolic/isodiastolic inequalities investigated in \cite{Abbalii} and \cite{AmbMon}. Theorem D, in particular, characterises curves attaining the absolute maximum of the scaling invariant quotient $\mathcal{S}(\Gamma)/L(\Gamma)$. In contrast, no Riemannian surface with connected, strictly convex boundary is such that its boundary is a local maximum of $\mathcal{S}/L$, see Proposition \ref{propnocriticalSL}. 

	In Section \ref{introcomparison}, we discuss relations between the above three geometric invariants of $\Gamma$ and other min-max geometric quantities. Before doing that, let us derive an immediate consequence of the results of this section, regarding the set of critical points of $\mathcal{D}$.

\subsection{On the number of critical points} Let $\Gamma$ be a smoothly embedded circle in a complete Riemannian manifold. Combining Theorems A and C, we can confirm that the functional $\mathcal{D}$ on $\mathcal{P}_*$ enjoys similar variational properties to the smooth maps $\Phi:[v]\in \mathbb{RP}^2\mapsto \langle A(v),v\rangle/|v|^2 \in \mathbb{R}$, where $A:\mathbb{R}^3\rightarrow\mathbb{R}^3$ is a linear self-adjoint operator with one-dimensional kernel and non-negative eigenvalues. (\textit{Cf.} \cite{LusSch}, Chapter $4$, $\S 7$, example 2). 

\begin{thmE}
	Let $\Gamma$ be a smoothly embedded circle in a complete Riemannian manifold. Let $\mathcal{D}$ be the distance function on pairs of points of $\Gamma$. Then, there are two possibilities:
	\begin{enumerate}
		\item $0<\mathcal{S}(\Gamma)<diam(\Gamma)$, in which case $\mathcal{D}$ has at least two non-trivial critical points in $\mathcal{P}_*$.
		\item $\mathcal{S}(\Gamma)=diam(\Gamma)$, in which case $\mathcal{D}$ has infinitely many non-trivial critical points in $\mathcal{P}_*$.
	\end{enumerate}
	Assume, moreover, that $\Gamma$ is the boundary of a smoothly embedded totally convex disc. Then, in case $(2)$, the set of all non-trivial critical points of $\mathcal{D}$ is a continuously embedded, homotopically non-trivial circle in $\mathcal{P}_*$.
\end{thmE}

	The estimate on the minimal number of non-trivial critical points is sharp, as the ellipses show. 

\subsection{An application: involutive symmetry} We would like to illustrate how the concepts and methods introduced in this paper lead very naturally to generalisations of classical results about plane curves of constant width to a Riemannian setting.

	A very simple result about plane curves of constant width is that circles centred at the origin are the only such curves that are invariant by the involution $x \mapsto -x$. The generalisation we propose reads as follows:

\begin{thmF}
	Let $\Omega$ be a totally convex disc with smooth boundary in a complete Riemannian surface $(M^2,g)$. Assume that every two points of $\partial \Omega$ are joined by a unique geodesic.\\ 
	\indent If $(\Omega,g)$ admits an isometric involution with no fixed boundary points and
	\begin{equation*}
		\mathcal{S}(\partial \Omega)=diam(\partial \Omega),
	\end{equation*}
	then the involution has a unique fixed point $x_0\in \Omega$, $\mathcal{S}(\partial \Omega)/2$ is not bigger than the injectivity radius of $(M^2,g)$ at $x_0$, and $ \Omega$ is the geodesic ball of $(M^2,g)$ of diameter $\mathcal{S}(\partial \Omega)=diam(\partial \Omega)$ and center $x_0$.
\end{thmF}

	The assumption on the uniqueness of geodesics is satisfied by convex discs in Cartan-Hadamard surfaces and by convex discs in a hemisphere of an Euclidean sphere, for instance. The converse of Theorem F is discussed in Example \ref{examthmF} in Section 8. It is interesting to observe that there are discs satisfying all the hypotheses of Theorem F, except the property that every two points of $\partial \Omega$ are joined by a unique geodesic, and which are not rotationally symmetric, nor geodesic balls, see Example \ref{exammodifiedellipsoid} in Section \ref{sectionexamples}.
	
	On the other hand, there are plane curves of constant width with reflection symmetries. There are also plane curves of constant width without any non-trivial symmetries at all \cite{Fil}. In other words, among Riemannian discs with convex boundary such that the width of the boundary equals the diameter of the boundary, those described by Theorem F form a special class. 

\subsection{Comparison between min-max invariants} \label{introcomparison} Let $(\Omega,g)$ be a totally convex disc with smooth boundary in a complete Riemannian surface. We proved that if $\Omega$ is a convex plane region, then $\mathcal{S}(\partial \Omega)$ coincides with the width of the boundary curve in the sense of the narrowest slab containing this curve, see Corollary \ref{corcoincidencewidths}. There are other min-max quantities that also generalise the classical notion in this sense. Let us describe two of them.

	A min-max construction to prove the existence of free boundary geodesics was considered by Xin Zhou in \cite{Zho}. Define the number
\begin{equation*}
    E_{\ast} = \inf_{\{\alpha_t\}}\,\max_{t\in [0,1]}\, \int_{0}^1 |\alpha_t^{\prime}(u)|^2du,
\end{equation*}
where $\{\alpha_t\}$ is a continuous path of $W^{1,2}$ maps $\alpha_t : [0,1] \rightarrow{\Omega}$ such that
\begin{itemize}
    \item[$i)$] the extremities $\alpha_t(0)$ and $\alpha_t(1)$ belong to $\partial \Omega$;
    \item[$ii)$] the map $(t,u)\in [0,1]^2 \mapsto \alpha_t(u)\in \Omega$ is continuous;
    \item[$iii)$] the maps $\alpha_0$ and $\alpha_1$ are constant maps; and
    \item[$iv)$] $\{\alpha_t\}$ is homotopic to a fixed path $\{\overline{\alpha}_t\}$ with all the above properties.
\end{itemize}
	
	Zhou proved that, if $E_{\ast}>0$, then $E_*$ is realised as the energy of a free boundary geodesic $\gamma : [0,1] \rightarrow \Omega$. In particular, its length $L(\gamma) = \sqrt{E_{\ast}}$ can be regarded as a min-max invariant $w_{\ast}$, which is a critical value of the length functional. In the same article, Zhou also considered the more general setting in which $\partial \Omega$ is replaced by a closed submanifold of general dimension and codimension, and there is no convexity assumptions. (For related work, see also \cite{GluZil}, \cite{Wei}, \cite{NabRot} and \cite{MLi}).
	
	The free boundary setting of the Almgren-Pitts min-max theory for curves on surfaces was investigated by Donato \cite{Don} and Donato and the second named author \cite{DonMon}. Consider the min-max invariant
\begin{equation*}
    \omega(\Omega,g) = \inf_{\{c_t\}} \sup_{t\in [0,1]} L(c_t),
\end{equation*}
where the infimum is considered over the paths $\{c_t\}$ of relative flat cycles modulo 2, which are homotopically non-trivial loops (with no base point fixed) in the space of relative cycles modulo 2. This definition allows for more sweepouts of $\Omega$ than those considered in \cite{Zho}. For instance, $c_t$ can be a network of curves with extremities that lie either in $\partial \Omega$ or at interior junctions at which an even number of curves meet. In particular, $c_t$ could be a closed embedded curve. It is proven in \cite{DonMon} that, if the Gaussian curvature of $(\Omega,g)$ is non-negative and its boundary is strictly convex, then $\omega(\Omega,g)$ is realised either as the length of a free boundary geodesic, or as the length of a geodesic loop whose vertex lies at the boundary of $\Omega$. (See \cite{DonMon}, Theorem 1.1). Moreover, while $\omega(\Omega, g) < L(\partial \Omega)$ always hold, examples show that $\omega(\Omega,g)$ can be arbitrarily close to the boundary length (see \cite{DonMon}, Section 6).

	When $\Omega$ is a strictly convex planar domain, all the three numbers $\mathcal{S}(\partial \Omega)$, $\omega(\Omega,euc)$ and $w_*$ coincide. If the Gaussian curvature of $\Omega$ is non-negative and its boundary is strictly convex, then $\mathcal{S}(\partial \Omega)\leq \omega(\Omega, g) \leq w_{\ast}$. If moreover $\mathcal{S}(\partial \Omega) = \omega(\Omega, g)$, then $\mathcal{S}(\partial \Omega)=w_*$ as well, and these three numbers coincide with the length of some free boundary minimising geodesic. On the other hand, there are examples of such discs for which
\begin{equation}\label{equation-widths-different}
    \mathcal{S}(\partial \Omega) < \frac{L(\partial\Omega)}{2} < \omega(\Omega, g) < L(\partial \Omega) < w_{\ast},
\end{equation}
see Figure 1. The justification for all these assertions is given in Section \ref{sectioncompare}.

\begin{figure}[htp]
    \centering
    \includegraphics[width=7cm]{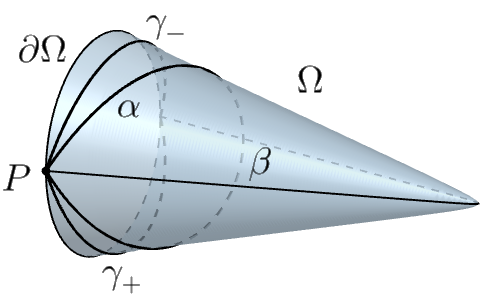}
    \caption{The curves realising different min-max widths.}
    \label{Figura-loop}
\end{figure}
	
	These assertions show that, among these min-max invariants, $\mathcal{S}(\partial \Omega)$ is the only one that is a suitable generalisation of the width of convex plane curves for the formulation of results like those in discussed above in Subsection 1.7.

\subsection{Perspectives} It would be interesting to know whether some refinement of the min-max Theorems A and B could prove that smoothly embedded totally convex discs always contain either a free boundary minimising geodesic with free boundary index at most one, or a pair of minimising geodesics, with extremities on the boundary, that are simultaneously stationary. (\textit{Cf}. Corollary \ref{corexistence}).

	It would also be interesting to see a systematic development of the Morse-Lusternik-Schnirelman theory of the distance function between pairs of points of compact embedded submanifolds of Riemannian manifolds, and a systematic investigation of the geometric meaning of the width-like invariants appearing in such theory in comparison to other geometric invariants.

	We believe that the methods of this paper can be generalised to the functional that computes the area of solutions to the Plateau problem for a given Jordan curve in a sphere inside a Riemannian three-manifold. We intend to describe this generalisation in a future work.

\subsection{Plan for the paper} In Section \ref{sectionbasicproperties}, we develop the basic theory of regular and critical points of the distance functional $\mathcal{D}:\mathcal{P}_*\mapsto [0,+\infty)$ along the lines of \cite{GroShi} and \cite{Gro}. (See also the excellent exposition of this theory by W. Meyer \cite{Mey}. We point out that a related, but different notion of critical pairs of points of the distance function has been investigated in \cite{Ald}). In particular, we prove the existence of gradient-like flows on the set of regular points (Proposition \ref{propgradientlikeflow}). We also prove that critical points satisfy special properties when the embedded circle bounds a totally convex smoothly embedded disc, see in particular Proposition \ref{propuniqu} and \ref{proplocalminima}.

	Building on these preliminary results, we prove in Section \ref{sectiongeneralminmaxtheorem} the basic min-max theorem asserting that the width $\mathcal{S}$ is a critical value of $\mathcal{D}$ (Theorem \ref{thmAbis}).
	
	The basic min-max Theorem A is improved in Section \ref{sectiondiscs}, under the assumption that the circle is the boundary of a totally convex disc. The key argument is to show that, under the extra assumption that no free boundary stable minimising geodesic exist in the disc, every critical point belongs to a sweepout restricted to which the functional attains a strict maximum exactly at this critical point, and moreover this sweepout is ``monotone" near this critical point (Propositions \ref{propstarcritical1} and \ref{propstarcritical2}). Birkhoff's curve shortening process, adapted to curves with boundary as in \cite{GluZil}, \cite{Zho} and \cite{DonMon}, is used here as a convenient technical device. This leads to Theorem B (Theorem \ref{thmBbis}).
	
	Section \ref{sectionzoll} contains the characterisations of the equalities between the width and the diameter (Theorem \ref{thmCbis}), the diameter and half the length (Lemma \ref{lemdiamlength}), and the width and half the length of a smoothly embedded circle (Theorem \ref{thmDbis}). The combination of Theorem A and Theorem C yields the proof of the sharp estimate for the number of critical points of $\mathcal{D}$, Theorem E, which is explained at the end of Section \ref{sectionzoll}. 
	
	Section \ref{sectioninvolutive} is dedicated to the proof of Theorem F (Theorem \ref{thmFbis}), and Section \ref{sectioncompare} to the comparison between min-max quantities associated to Riemannian discs with non-negative Gaussian curvature and strictly convex boundary. Finally, Section \ref{sectionexamples} contains the several examples which illustrate different aspects of what was discussed in the rest of the paper. 
	
\section{Basic properties of regular and critical points} \label{sectionbasicproperties}

	Let $\Gamma$ be a smoothly embedded circle in a complete Riemannian manifold $(M^n,g)$. We denote by $\mathcal{P}$ the set consisting of subsets $\{x,y\}$ of $\Gamma$ with at most two points, endowed with the obvious topology with respect to which $\{x_k,y_k\}\rightarrow \{x,y\}$ if and only if $x_{k}\rightarrow x$ and $y_{k}\rightarrow y$. Recall that $\mathcal{P}$ is a smooth surface with boundary $\partial \mathcal{P}=\{\{p\}\in \mathcal{P}\,|\,p\in \Gamma\}$, diffeomorphic to a Moebius band. We identify $\mathcal{X}\in T_{\{p,q\}}\mathcal{P}$ with sets $\{\mathcal{X}(p),\mathcal{X}(q)\}$ where $\mathcal{X}(p)\in T_p \Gamma$ and $\mathcal{X}(q)\in T_q \Gamma$. The surface $\mathcal{P}\setminus \partial \mathcal{P}$ inherits a Riemannian metric so that $\langle \mathcal{X}, \mathcal{Y} \rangle(\{p,q\})= \langle \mathcal{X}(p),\mathcal{Y}(p)\rangle + \langle \mathcal{X}(q),\mathcal{Y}(q)\rangle$.
	
	The quotient of $\mathcal{P}$ by the identification of all boundary points as a single point $0_*=[\{x,x\}]$ is a projective plane that we denote $\mathcal{P}_*$. \\
\indent We assume non-trivial geodesics are parametrized by arc length. When $\gamma$ joins $p$ and $q\in \Gamma$, the orthogonal projections of outward-pointing conormal vectors of $\gamma$ onto $T\Gamma$ define an element $\nu^T_\gamma=\{\nu_\gamma^{T}(p),\nu_\gamma^T(q)\}$ of $T\mathcal{P}$. If $\gamma:[0,a]\rightarrow M$ is such that $\gamma(0)=p$ and $\gamma(a)=q$, then
\begin{equation*}
	\nu_\gamma(p)=-\gamma'(0) \quad \text{and} \quad \nu_\gamma(q)=\gamma'(a).
\end{equation*}

	Denote by $d(x,y)$ the Riemannian distance between points $x$ and $y$ in $(M^n,g)$. We collect below some basic lemmas concerning the ``Morse Theory" of the distance functional,
\begin{equation*}
	D : \{p,q\}\in \mathcal{P} \mapsto d(p,q)\in [0,+\infty),
\end{equation*}
developed according to the notions of regular and critical points introduced in Definition \ref{defcriticalpoints}.

	We begin with the statement of a simple compactness result.
\begin{lemm}\label{lemmabasiccompactness}
	If $\{p_i,q_i\}\subset \Gamma$ converges to $\{p,q\}\subset \Gamma$, then any sequence of minimising geodesics $\gamma_i$ joining $p_i$ and $q_i$ has a converging subsequence to a minimising geodesic $\gamma$ joining $p$ and $q$.
\end{lemm}

	Recall that regular points $\{x,y\}\subset \Gamma$ of $\mathcal{D}$ are such that $x\neq y$.  It is not difficult to check that if $x$ and $y$ are different points of $\Gamma$ that lie in some geodesically convex ball of $(M^n,g)$ and, moreover, bound a very small arc of $\Gamma$, then $\{x,y\}\subset \Gamma$ is a regular point of $\mathcal{D}$. Thus, there exists $\varepsilon>0$ such that no critical point of $\mathcal{D}$ exists on $\{\{x,y\}\in \mathcal{P}\,|\,0<d(x,y)<\varepsilon\}$.

\begin{lemm}\label{lemmopenandtheta}
	Suppose $\{p,q\}\subset \Gamma$ is a regular point of $\mathcal{D}$. Then, there exist disjoint open neighbourhoods $U_1$ of $p$ and $U_2$ of $q$ in $\Gamma$ with the following property: there exist $\theta>0$ and vector fields $v_1$ on $U_1$ and $v_2$ on $U_2$, that are tangent to $\Gamma$, such that
	\begin{equation}\label{eqregularcomtheta}
		 \langle v_1(x),\nu_\gamma^{T}(x)\rangle + \langle v_2(y),\nu_\gamma^{T}(y)\rangle \leq - \theta 
	\end{equation} 
	for every minimising geodesic $\gamma$ joining a pair $\{x,y\}$ with $x\in U_1$ and $y\in U_2$. 
\end{lemm}
\begin{proof}
	By Definition \ref{defcriticalpoints} and the compactness of minimising geodesics bounding the regular point $\{p,q\}$ (Lemma \ref{lemmabasiccompactness}), it follows that there exists a number $\theta>0$ and a vector $(v_1,v_2)\in T_p\Gamma\times T_q\Gamma$ such that 
	\begin{equation*}
		 \langle v_1,\nu_\gamma^{T}(p)\rangle + \langle v_2,\nu_\gamma^{T}(q)\rangle \leq - 2\theta 
	\end{equation*} 
	for every minimising geodesic $\gamma$ joining $p$ and $q$.  Extend $v_1$ and $v_2$ smoothly and tangentially to $\Gamma$ in some small neighbourhood of $p$ and $q$. (Notice that $p\neq q$, so that there is no ambiguity). Then, using again Lemma \ref{lemmabasiccompactness}, possibly after taking smaller neighbourhoods, we find an open neighbourhood of $\{p,q\}$ in $\mathcal{P}$ in such way that the vector fields $v_1$ and $v_2$ satisfy
		\begin{equation*}
		  \langle v_1(x),\nu_\gamma^{T}(x)\rangle + \langle v_2(y),\nu_\gamma^{T}(y)\rangle \leq - \theta 
	\end{equation*} 
	for all $\{x,y\}$ in this neighbourhood of $\{p,q\}$ in $\mathcal{P}$ and for all minimising geodesics $\gamma$ joining such pairs $\{x,y\}$. 
\end{proof}

\begin{cor}\label{corcriticalcompactness}
	The critical points of $\mathcal{D}$ form a compact subset of $\mathcal{P}$.
\end{cor}
\begin{proof}
	In fact, by the Lemma \ref{lemmopenandtheta}, its complement (the set of regular points) is an open subset of $\mathcal{P}$.
\end{proof}

\indent From now on, we denote by $\mathcal{U}$ the set of regular points of $\mathcal{D}$ in $\mathcal{P}$. 

\begin{lemm}\label{lemquasigradient}
	Let $\mathcal{U}\subset \mathcal{P}$ be the open subset of regular points of $\mathcal{D}$. There exists a unit vector field $\mathcal{X}\in T\mathcal{U}$ such that $\langle \mathcal{X},\nu_\gamma^T\rangle < 0$ for every minimising geodesic $\gamma$ joining a pair of points $\{x,y\}\in\mathcal{U}$.
\end{lemm}

\begin{proof}
	In a finite dimensional vector space with an inner product, linear combinations $w=\lambda_1w_1+\ldots+\lambda_kw_k$ with $\lambda_i\in[0,1]$ and $\sum\lambda_i=1$ of vectors $w_i$ that satisfy $\langle w_i,v \rangle < 0$ for all $i$ also satisfy $\langle w,v\rangle<0$. Thus, we may glue the vector fields $V$ constructed in Lemma \ref{lemmopenandtheta} together, by means of a partition of unity subordinated to some locally finite cover of $\mathcal{U}$ in $\mathcal{P}$ consisting of open sets as in Lemma \ref{lemmopenandtheta}, and normalise the resulting vector field on $\mathcal{U}$ so to obtain a unit vector field on $\mathcal{U}$ with the desired properties.
\end{proof}

\begin{prop}\label{propgradientlikeflow}
	Let $\mathcal{X}$ be a unit vector field on $\mathcal{U}\subset \mathcal{P}$ as in Lemma \ref{lemquasigradient} and let $\mathcal{K}$ be a compact subset of $\mathcal{U}$. Denote by $\phi_t$ the flow of $\mathcal{X}$ starting at points of $\mathcal{U}$. Then, there exists $\theta>0$ such that, for every pair of points $\{x,y\}\in \mathcal{K}$,
	\begin{equation*}
		\mathcal{D}(\phi_t(\{x,y\})) \leq d(x,y)-\theta t
	\end{equation*}	    
	 and
	\begin{equation*}
		\mathcal{D}(\phi_{-t}(\{x,y\}) \geq d(x,y)+\theta t
	\end{equation*}	    
as long as the flow starting at $\{x,y\}\in \mathcal{K}$ exists in $\mathcal{K}$.
\end{prop}

\begin{proof}
	Fix some compact $\mathcal{L}\subset \mathcal{U}$ whose interior contains the given compact subset $\mathcal{K}\subset \mathcal{P}$. The flow $\phi_t$ of $\mathcal{X}$ (and of $-\mathcal{X}$) starting at any point $\{x,y\}\in \mathcal{K}$ exists for a uniformly positive time and $\phi_t(\{x,y\})$ remains in $\mathcal{L}$ for this short duration of time. As a consequence of Lemma \ref{lemquasigradient} and the compactness of $\mathcal{L}\subset \mathcal{U}$, there exists $\theta>0$ such that $\langle \mathcal{X},\nu^T_\gamma\rangle \leq -\theta$ for every $\{x,y\}\in \mathcal{L}$ and every minimising geodesic $\gamma$ joining $x$ and $y$.
	
	Given $\{p,q\}\in \mathcal{K}$, we write $\phi_t\{p,q\}=\{p_t,q_t\}\subset \Gamma$ for all $t$ contained in the interval of existence of the flow starting at $\{p,q\}$. Notice that $p_t$ and $q_t$ are smooth functions of $t$. (There is no ambiguity since $p_t\neq q_t$ as long as the flow exists in $\mathcal{U}$). The proof of the Lemma will be finished as soon as we check that the continuous function $h(t)=d(p_t,q_t)$, defined for those values of $t$ such that $\{p_t,q_t\}$ lies in $\mathcal{K}$, satisfies the differential inequality $h'\leq -\theta$ in the sense of support functions.
	
	In order to check this claim, we proceed as follows. Given $t_0$ in the domain of $h$, choose once and for all some minimising geodesic $\gamma$ of $(M^n,g)$ joining $p_{t_0}$ and $q_{t_0}$, and some smooth variation $\gamma_t$ of $\gamma_0=\gamma$ by curves with extremities $p_{t_0+t}$ and $q_{t_0+t}$. Then, define the smooth function $\hat{h}(t)=L(\gamma_t)$ for all $t$ sufficiently small so that the definition makes sense.
	
	Since $\gamma=\gamma_0$ is minimising, 
	\begin{equation*}
		\hat{h}(0)=L(\gamma)=d(p_{t_0},q_{t_0})=h(t_0).
	\end{equation*}
	Since $\gamma_t$ is a curve in $(M^n,g)$ joining the points $p_{t_0+t}$ and $q_{t_0+t}$, we have $\hat{h}(t)=L(\gamma_t)\geq d(p_{t_0+t},q_{t_0+t})=h(t_0+t)$. Thus, $\hat{h}$ is a support function for $h$ at $t_0$, which moreover satisfies
	\begin{equation*}
		\hat{h}'(0)= \frac{d}{dt}_{|_{t=0}} L(\gamma_t) = \Big\langle \frac{d}{dt}_{|_{t_0}}p_t,\nu_\gamma(p_{t_0})\Big\rangle+\Big\langle \frac{d}{dt}_{|_{t_0}}q_t,\nu_\gamma(q_{t_0})\Big\rangle
	\end{equation*}	  
by the first variation formula \eqref{eqfirstvariation} for the geodesic $\gamma=\gamma_0$. By construction, $\mathcal{X}(\{p_{t_0},q_{t_0}\})=\frac{d}{dt}_{|_{t=t_0}}\phi_t(\{p_t,q_t\})=\{\frac{d}{dt}_{|_{t_0}}p_t,\frac{d}{dt}_{|_{t_0}}q_t\}$. Hence, 
	\begin{equation*}
		\hat{h}'(0) = \langle \mathcal{X}(p_{t_0}),\nu_\gamma^T(p_{t_0})\rangle+\langle \mathcal{X}(q_{t_0}),\nu_\gamma^T(q_{t_0})\rangle\leq -\theta.
	\end{equation*}
   The claim follows, and this finishes the proof of the proposition.
\end{proof}

	As an immediate corollary, we have that boundary points at distance $diam(\Gamma)$ are critical points of $\mathcal{D}$. More generally:

\begin{cor}\label{cormaximumiscritical}
	Any local maximum or local minimum $\{p,q\}\subset \Gamma$ of $\mathcal{D}$ is a critical point of $\mathcal{D}$.
\end{cor}

	To finish this section, we discuss some properties of minimising geodesics bounding a pair $\{x,y\}\subset \Gamma$ of critical points of $\mathcal{D}$ when $\Gamma$ bounds a totally convex disc $\Omega$ in a complete Riemannian surface.
	
	Two-dimensional discs are special from a topological point of view because embedded curves joining boundary points divide $\Omega$ into two components, by Jordan Theorem. We always orient it so that the induced orientation on $\partial \Omega$ is the counter-clockwise orientation.
	
	Since minimising geodesics joining the same two points cannot intersect except at their extremities, the topological observation above implies that, for every $\{p,q\}\subset \Gamma$ with $p\neq q$, there exists (possibly equal) minimising geodesics $\gamma_+$ and $\gamma_-$, with extremities $p$ and $q$, such that every other minimising geodesic between $p$ and $q$ lie in between the region bounded by $\gamma_+$ and $\gamma_-$. The pair $\gamma_{+}$ and $\gamma_{-}$ is unique up to relabelling. We call them the \textit{extremal minimising geodesics} of the pair $\{p,q\}$. 
	
	Of course, $p$ and $q$ are joined by a unique minimising geodesic if and only if $\gamma_+=\gamma_-$. Furthermore, $\gamma_+$ or $\gamma_{-}$ may be in fact arcs of $\partial \Omega$. Notice that the only scenario where both extremal minimising geodesics bounded by $p$ and $q$ are arcs of $\partial \Omega$ is when $\partial \Omega$ is a closed geodesic formed by two minimising arcs with extremities $p$ and $q$. 

\begin{prop}\label{propextremalgeodesics}
	Let $(\Omega,g)$ be a totally convex disc with smooth boundary in some complete Riemannian surface, and denote by $T$ the unit tangent vector field to $\partial \Omega$ that is compatible with its orientation.
	
	Let $\{p,q\}\subset \partial \Omega$ be a non-trivial critical point of $\mathcal{D}$, and denote by $\gamma_{+}$ and $\gamma_{-}$ the extremal minimising geodesics joining $p$ and $q$, labelled in such way that $\langle \nu_{\gamma_+}(p),T(p)\rangle\leq \langle \nu_{\gamma_-}(p),T(p)\rangle$.
	
	Then
	\begin{equation*}
		\langle \nu_{\gamma_+}(p),T(p)\rangle \leq 0 \leq \langle \nu_{\gamma_-}(p),T(p)\rangle
	\end{equation*}	
	and
	\begin{equation*}
		\langle \nu_{\gamma_-}(q),T(q)\rangle \leq 0 \leq \langle \nu_{\gamma_+}(q),T(q)\rangle.
	\end{equation*}
\end{prop}   
\begin{proof}
If $\gamma_{-}=\gamma_{+}$, then there is only one minimising geodesic joining $p$ and $q$, and it must be a free boundary minimising geodesic because $\{p,q\}\subset \partial \Omega$ is a critical point of $\mathcal{D}$. In this case, there is nothing else to be proven.

	From now on, we therefore assume $\gamma_{+}\neq\gamma_{-}$. Observe that since all minimising geodesics joining $p$ and $q$ must lie in the region of $\Omega$ bounded between $\gamma_{-}$ and $\gamma_{+}$, every minimising geodesic $\gamma$ joining these two points is such that the number $\langle \nu_{\gamma},T(p)\rangle$ belongs to the closed interval between the numbers $\langle \nu_{\gamma_{+}},T(p)\rangle$ and $\langle \nu_{\gamma_{-}},T(p)\rangle$ (which is contained in $[-1,1]$). For, otherwise, the minimising geodesic $\gamma$ issues from $p$ enters the complement of the region bounded between $\gamma_{-}$ and $\gamma_{+}$ and must stay there until it reaches $q$, because it cannot cross either $\gamma_{+}$ or $\gamma_{-}$. But this contradicts the extremal property of these two minimising geodesics. A similar assertion holds, of course, at the point $q$. \\
	
\noindent \textbf{Claim}: The interval between the numbers $\langle \nu_{\gamma_+}(p),T(p)\rangle$ and $\langle \nu_{\gamma_-}(p),T(p)\rangle$ contains zero, and the interval between the numbers $\langle \nu_{\gamma_+}(q),T(q)\rangle$ and $\langle \nu_{\gamma_-}(q),T(q)\rangle$ contains zero.\\

	Suppose not. Then, at the point $p$ we have $\langle \nu_{\gamma_+}(p),T(p)\rangle<0$ and $\langle \nu_{\gamma_-}(p),T(p)\rangle<0$, say. It follows from the preceding observations that the vector  $(v_1,v_2)=(T(p),0)\in T_{p}\partial \Omega\times T_q\partial \Omega$  is such that $\langle v_1,\nu_{\gamma}^T\rangle + \langle v_2,\nu_{\gamma}^T\rangle = \langle T(p),\nu_{\gamma}^T\rangle < 0$ for all minimising geodesics joining $p$ and $q$. But this contradicts the assumption that $\{p,q\}\subset \partial \Omega$ is a critical point of $\mathcal{D}$. Making obvious modifications, the same argument rules out all other possibilities except those described in the claim.
	
	Next, recall that we labelled $\gamma_+$ and $\gamma_{-}$ in such way that $\langle \nu_{\gamma_+}(p),T(p)\rangle\leq \langle \nu_{\gamma_-}(p),T(p)\rangle$ holds. Hence, it follows from the Claim that 
\begin{equation}\label{eqauxaux}
	\langle \nu_{\gamma_+}(p),T(p)\rangle \leq 0 \leq \langle \nu_{\gamma_-}(p),T(p)\rangle.
\end{equation}
The Claim also implies that there are two possibilities for the numbers $\langle \nu_{\gamma_\pm}(q),T(q)\rangle$. If $\langle \nu^{T}_{\gamma_{+}}(q),T(q)\rangle \leq 0 \leq \langle \nu^{T}_{\gamma_{-}}(q),T(q)\rangle$, then it it is easy to see that this angle conditions together with $\eqref{eqauxaux}$ imply that the different geodesics $\gamma_{+}$ and $\gamma_{-}$ cross each other at some point inside the disc $\Omega$. Since they are minimising geodesics joining $p$ and $q$, this is not possible. Therefore the other possibility holds, namely
\begin{equation*}
	\langle \nu_{\gamma_-}(q),T(q)\rangle \leq 0 \leq \langle \nu_{\gamma_+}(q),T(q)\rangle,
\end{equation*}
as we wanted to show.
\end{proof}

	Using Proposition \ref{propextremalgeodesics}, we can prove a uniqueness result that will play a key role in Sections \ref{sectiondiscs} and \ref{sectionzoll}.

\begin{prop}\label{propuniqu}
	Let $(\Omega,g)$ be a totally convex Riemannian disc with smooth boundary in a complete Riemannian surface. If $\{p,q\}$, $\{p,r\}\subset \partial \Omega$ are critical points of $\mathcal{D}$ with $p\neq q$ and $p\neq r$, then $q=r$.
\end{prop}
\begin{proof}
	Suppose, by contradiction, that $q\neq r$. Let $\{T,N\}\subset T_p \Omega$ be an orthonormal basis so that $N$ points inside $\Omega$ and $T$ defines the induced orientation of $T_p\partial \Omega$. Relabelling the points $q$ and $r$, if necessary, we may assume that $p$, $q$ and $r$ appears in this order as we move in the counter-clockwise direction along the boundary.
	
	Since $\{p,q\}$ is critical, one of its extremal minimising geodesics, say $\gamma_-$, will issue from $p$ with $\langle\gamma'_-(0),T\rangle \leq 0$, by Proposition \ref{propextremalgeodesics}. Similarly, since $\{p,r\}$ is also critical, one of its extremal minimising geodesics, say $\gamma_{+}$, will issue from $p$ with $\langle\gamma'_+(0),T\rangle \geq 0$, by Proposition \ref{propextremalgeodesics}.
		
	By the relative position of $r$ with respect to $p$ and $q$, the point $r$ does not lie in the component $A$ of $\Omega\setminus\gamma_{-}$ that contains the (oriented) arc of $\Gamma$ between $p$ and $q$. Similarly, the point $q$ does not lie in the component $B$ of $\Omega\setminus\gamma_{+}$ that contains the (oriented) arc of $\Gamma$ between $r$ and $p$. 	
	
	Since $\gamma_{-}$ and $\gamma_{+}$ are not equal (because we are assuming by contradiction that their extremities $q$ and $r$ are not equal), their tangent vectors at $p$ are not equal. Since $\langle\gamma'_-(0),T\rangle \leq 0 \leq \langle\gamma'_+(0),T\rangle$, there are two possibilities. Either $0<\langle\gamma'_+(0),T\rangle$ so that $\gamma_{+}'(0)$ points inside $A$, in which case $\gamma_+$ must cross $\gamma_{-}$ somewhere in $\Omega$ before reaching $r\in \Omega\setminus \overline{A}$, a contradiction. Or $\langle\gamma'_-(0),T\rangle < 0$ so that $\gamma_{-}'(0)$ points inside $B$, and a contradiction arises similarly. The proposition follows.
\end{proof}

\indent We remark that the distance function $d_p$ to a given boundary point $p$, when restricted to $\partial \Omega$, $(d_p)_{|_{\partial \Omega}}$, may have more than one critical point in $\partial \Omega$, say $q\neq q'$. For instance, consider a smooth plane convex curve that contains an arc of circle centred at some point $p$ of the curve. What Proposition \ref{propuniqu} says is that, when $(\Omega,g)$ is a totally convex disc with \textit{smooth} boundary in a complete Riemannian surface, then such point $p$ cannot be part of two different non-trivial critical points $\{p,q\}, \{p,q'\}\subset \partial \Omega$ of the distance function $\mathcal{D}$. (The Reuleaux triangle shows that some smoothness assumption is essential).

\indent To finish this section, we use the same Proposition \ref{propextremalgeodesics} to characterise the minimising geodesics bounding a non-trivial local minimum of $\mathcal{D}$.

\begin{prop}\label{proplocalminima}
	Let $(\Omega,g)$ be a totally convex Riemannian disc with smooth boundary in a complete Riemannian surface. If $\{p,q\}\subset \partial \Omega$ is a non-trivial local minimum of $\mathcal{D}$, then there exists only one minimising geodesic joining $p$ and $q$, and this minimising geodesic is moreover free boundary stable.
\end{prop}
\begin{proof}
	Let $\gamma_{+}$ and $\gamma_{-}$ be the extremal minimising geodesics joining $p$ and $q$, labelled as in Proposition \ref{propextremalgeodesics}. We will first show, by a contradiction argument, that $\gamma_{+}=\gamma_{-}$.
	
	If $\gamma_{+}\neq \gamma_{-}$, then at least one of the inequalities in Proposition \ref{propextremalgeodesics} is strict. Let us say $\gamma_+$ satisfies $\langle \nu_{\gamma_{+}}(p),T(p)\rangle<0$. (The other cases are analogous). Then, by the first variation formula \eqref{eqfirstvariation}, any vector field tangent to $\partial \Omega$ extending $T(p)$ with support in a sufficiently small neighbourhood of $p$ that do not contain $q\neq p$ generates a variation of $\gamma_{+}$ by curves with extremities of the form $\{x,q\}\subset \partial \Omega$, for some $x\in \partial \Omega$ arbitrarily close to $p$, that are strictly shorter than $\gamma_{+}$. But then there are points $\{x,q\}$ arbitrarily close to $\{p,q\}$ with $d(x,q)<L(\gamma_+)=d(p,q)$, a contradiction with the minimising property of $\{p,q\}$.
	
	Thus, as claimed, the extremal minimising geodesics are the same geodesic. Since $\{p,q\}$ is a non-trivial critical point of $\mathcal{D}$, the unique minimising geodesic $\gamma_{+}=\gamma_{-}$ joining them meets $\partial \Omega$ orthogonally. Since $\{p,q\}$ is a local minimum of $\mathcal{D}$, it follows easily from the second variation formula \eqref{eqsecondvariation} that $\gamma_{+}=\gamma_{-}$ is free boundary stable.
\end{proof}

\section{The basic min-max theorem} \label{sectiongeneralminmaxtheorem}

\indent The width of a smoothly embedded circle $\Gamma$ in a complete Riemannian manifold is the number   
\begin{equation*}
	\mathcal{S}(\Gamma) = \inf_{\{p_t,q_t\}\subset \mathcal{V}} \max_{t\in [0,1]} d(p_t,q_t),
\end{equation*}
where $\mathcal{V}$ is the set of sweepouts $\{p_t,q_t\}$ of $\Gamma$, see Definition \ref{defiweepout}. The key property of sweepouts is that, when they are regarded as continuous paths in the projective plane $\mathcal{P}_{*}$ associated to $\Gamma$, they are homotopically non-trivial loops based on the absolute minimum of $\mathcal{D}$. Moreover, isotopies of $\mathcal{P}_*$  that fix a neighbourhood of the point $0_{*}=[\{p\}]\in \mathcal{P}_*$ deform any sweepout into another sweepout.

	Using the gradient-like flow constructed on compact subsets of the set of regular values of $\mathcal{D}$ in $\mathcal{P}$ (Proposition \ref{propgradientlikeflow}), it is then easy to deduce Theorem A. (\textit{Cf}. \cite{Mil}). 
\begin{thm}\label{thmAbis}
	Let $\Gamma$ be a smoothly embedded circle in a complete Riemannian manifold. Then $\mathcal{S}(\Gamma)>0$ and there exists a critical point  $\{p,q\} \subset \Gamma$ of $\mathcal{D}$ such that
	\begin{equation*}
		d(p,q)=\mathcal{S}(\Gamma).
	\end{equation*}
\end{thm}
\begin{proof}
	As seen in Section \ref{introwidth}, there exists $\varepsilon>0$ such that $\mathcal{S}(\Gamma)>2\varepsilon$. Suppose, by contradiction, that there exists $\delta>0$ such that the compact subset 
	\begin{equation*}
		\mathcal{K}=\{\{x,y\}\in \mathcal{P}\,|\,\mathcal{S}(\Gamma)-\varepsilon\leq d(x,y)\leq \mathcal{S}(\Gamma)+\delta\}
	\end{equation*}
	contains no critical points of $\mathcal{D}$. Multiply the vector field $\mathcal{X}$ from Lemma \ref{lemquasigradient} by an appropriate smooth non-negative cut-off function $\rho$ with $|\rho|\leq 1$ that equals $1$ on $\mathcal{K}$ and has support on a sufficiently small open neighbourhood of $\mathcal{K}$, which is still contained in the set of regular points, but is disjoint from $\{\mathcal{D}< \mathcal{S}(\Gamma)-2\varepsilon\}$. By Proposition \ref{propgradientlikeflow}, the flow of such vector field on $\mathcal{P}_*$ retracts $\{\mathcal{D}\leq\mathcal{S}(\Gamma)+\delta\}$ into $\{\mathcal{D}\leq \mathcal{S}(\Gamma)-\varepsilon\}$, in such way that no point of $\{\mathcal{D}< \mathcal{S}(\Gamma)-2\varepsilon\}$ moves. But then, a sweepout $\{p_t,q_t\}$ with $\max_{t\in [0,1]}d(p_t,q_t) < \mathcal{S}(\Gamma)+\delta$ (which exists, by definition of $\mathcal{S}(\Gamma))$ can be deformed continuously to another sweepout $\{\hat{p}_t,\hat{q}_{t}\}$ with $\max_{t\in [0,1]}d(\hat{p}_t,\hat{q}_t) < \mathcal{S}(\Gamma)-\varepsilon$, a contradiction with the definition of $\mathcal{S}(\Gamma)$. 
	
	Hence, the set $\mathcal{K}$ contains a critical point of $\mathcal{D}$. Since both $\delta$ and $\varepsilon$ can be made arbitrarily small, and the set of critical points is compact (Lemma \ref{corcriticalcompactness}), passing to the limit of a suitable subsequence of critical points we find a critical point $\{p,q\}$ of $\mathcal{D}$ such that $d(p,q)=\mathcal{S}(\Gamma)>0$.
\end{proof}

\section{The width of the boundary of discs}\label{sectiondiscs}

	In this section, we consider only embedded circles that are the boundary of a compact totally convex surface $(\Omega,g)$ with smooth boundary in some complete Riemannian surface. We will also work under the extra geometric assumption:
\begin{equation*}
	(\star) \textit{ no free boundary stable geodesic exists on } (\Omega,g).
\end{equation*}
It is well-known that this condition implies $\Omega$ to be diffeomorphic to a disc.
\begin{lemm}\label{lemmnostabelimpliesdisc}
	Let $(\Omega,g)$ be a compact totally convex region with smooth boundary in a complete Riemannian manifold. If no free boundary minimising geodesic is free boundary stable, then $\Omega$ is a disc.
\end{lemm}
\begin{proof}
	If there are at least two boundary components $\partial_1 \Omega$ and $\partial_2 \Omega$, then minimisation of the length of curves with one extremity in $\partial_1 \Omega$ and the other in $\partial_2 \Omega$ produces a proper minimising geodesic that meets $\partial_1 \Omega$ and $\partial_ 2 \Omega$ orthogonally, as can be seen by the first variation formula. If, as we are assuming, its index is not zero, then, by the second variation formula \eqref{eqsecondvariation}, some variation of this geodesic would be by curves, with one extremity in $\partial_1 \Omega$ and the other in $\partial_2 \Omega$, with strictly less length. But this a contradiction.
	
	If the surface $\Omega$ has just one boundary component but is not orientable, then its orientable double cover $\hat{\Omega}$ has two boundary components. Minimising the length among curves on $(\Omega,g)$ with extremities in $\partial\Omega$ that lift to a curve joining the two different components of $\partial \hat{\Omega}$ produces similarly a free boundary proper minimising geodesic with extremities in $\Omega$ that is free boundary stable, a contradiction.
	
	Finally, if $\Omega$ is orientable and has connected boundary and genus $g\geq 1$, then again minimisation of length among curves with extremities in $\partial \Omega$ and going trough the holes of $\Omega$ produces a proper minimising geodesic that is free boundary stable, a contradiction.
	
	Hence, the only topology that is not contradictory with the property described on the statement of the Lemma is the topology of an orientable genus zero surface with connected boundary. In other words, $\Omega$ is a disc. 
\end{proof}

	Without loss of generality, we assume from now on that $\Omega$ is a disc. As we will see, the importance of property $(\star)$ lies on the fact that it rules out non-trivial local minima of $\mathcal{D}$. (See Proposition \ref{proplocalminima}).

	Before we continue, we need to recall a few properties of unstable free boundary geodesics $\gamma$ in $(\Omega,g)$. By definition, a normal vector field $X$ exists along such $\gamma$ so that $Q(X,X)<0$. By virtue of the second variation formula \eqref{eqsecondvariation}, a small deformation of $\gamma$ by the flow of some extension of $X$ by a vector field that is tangent to $\partial \Omega$ is a deformation by curves in $\Omega$ with extremities in $\partial \Omega$ that are strictly shorter than $\gamma$.
	
	The next lemma refines this observation, in the sense that we may choose the vector field $X$ in such way that the curves move in a definite direction inside the disc $\Omega$.

\begin{lemm}\label{lemfirstjacobi}
	Let $\gamma:[0,a]\rightarrow \Omega$ be a free boundary geodesic joining points $p$ and $q$ in $\partial \Omega$, and let $\{\gamma'(s),X(\gamma(s))\}$ be an orthonormal frame along $\gamma$. 
If $\gamma$ is not free boundary stable, then there exists a smooth positive function $u$ on $\gamma$ such that $Q(uX,uX)<0$. 
\end{lemm}

\begin{proof}
	The index form of $\gamma$ can then be thought as acting on functions,
\begin{multline*}
	Q(\phi,\phi):= Q(\phi X,\phi X) = \int_{0}^{a} |\phi'(s)|^2 - K(\gamma(s))\phi^2(s) ds \\ - A(X(p),X(p))\phi^2(0) - A(X(q),X(q))\phi^2(a).
\end{multline*} 
	By standard techniques, the minimisation of $Q(\phi,\phi)/\int_{0}^a \phi(s)^2 ds$ among non-zero functions in $C^{\infty}([0,a])$ yields a solution $u\neq 0$ of
     \begin{align}\label{eqfirstjacobieigen}
     u'' + Ku + \lambda_1 u & =0 \quad \text{on} \quad (0,a), \nonumber \\
     -u'(0) &=  A(X(p),X(p))u(0), \\
     u'(a) &=  A(X(q),X(q))u(a). \nonumber
     \end{align} 
     for the constant 
     \begin{equation*}
     	\lambda_1 = \inf \Big\{\frac{Q(\phi,\phi)}{\int_{0}^a \phi(s)^2 ds}\,|\, \phi\in C^{\infty}([0,a]),\phi\neq 0\Big\}.
     \end{equation*}
     Also, conversely, any function $u\neq 0$ such that $Q(u,u)=\lambda_1\int_{0}^a u^2 ds$ is a non-trivial solution of \eqref{eqfirstjacobieigen} which is, moreover, either strictly positive or such that $-u$ is strictly positive on $[0,a]$.
     
     If $\gamma$ is not stable, then we clearly have $\lambda_1<0$, and the corresponding positive solution $u$ of \eqref{eqfirstjacobieigen} satisfies $Q(u,u)= \lambda_1 \int_0^a u^2(s)ds < 0$.
\end{proof}

	Let $\gamma_1$ and $\gamma_2$ be minimising geodesics on $(\Omega,g)$ with extremities $p_1$, $q_1$ and $p_2$, $q_2$, respectively. Assume that $p_1$, $p_2$, $q_2$ and $q_1$ are four different points that appear in that order as we traverse $\partial \Omega$ in the counter-clockwise direction. Because the geodesics are minimising, they cannot intersect each other. Thus, one of them belongs entirely to one of the components of the disc $\Omega$ that the other one defines.
	
	The next result shows that this configuration is not compatible with property $(\star)$. Thus, in a sense, the minimising geodesics bounding critical points of $\mathcal{D}$ on the boundary of Riemannian discs with property $(\star)$ satisfies some sort of ``\textit{Frankel property}", \textit{cf}. \cite{Fra}, Section 2, and \cite{FraLi}, Lemma 2.4.

\begin{prop}\label{propFrankel}
	Let $(\Omega,g)$ be a totally convex Riemannian disc with smooth boundary in a complete Riemannian surface. Assume $(\Omega,g)$ has property $(\star)$. If $\{p,q\}$, $\{r,s\}\subset \partial \Omega$ are non-trivial critical points of $\mathcal{D}$ such that $\{p,q\}\cap \{r,s\}=\emptyset$, then $r$ and $s$ lie in different connected components of $\partial \Omega\setminus \{p,q\}$. In particular, every minimising geodesic joining $p$ and $q$ meets every minimising geodesic joining $r$ and $s$.
\end{prop}
\begin{proof}
	Since $\Omega$ is a disc, it is enough to show that property $(\star)$ precludes the possibility that $r$ and $s$ lie in the same component of $\partial \Omega\setminus \{p,q\}$.
	
	Suppose, by contradiction, that $\{r,s\}$ lies in one of the boundary arcs determined by $\{p,q\}$. Renaming the points, we may assume that they appear in the following order $p\rightarrow r\rightarrow s\rightarrow q$, as we go around $\Gamma$ in the counter-clockwise direction.
	
	Let $\gamma_{+}$ be the extremal minimising geodesic joining $p$ and $q$ such that $\langle \nu_{\gamma_{+}}(p),T(p)\rangle\leq 0 \leq \langle \nu_{\gamma_{+}}(q),T(q)\rangle$, and let $\gamma_{-}$ the extremal minimising geodesic joining $r$ and $s$ such that $\langle \nu_{\gamma_{-}}(r),T(r)\rangle\geq 0\geq \langle \nu_{\gamma_{-}}(s),T(s)\rangle$. (Such angle conditions are met, recall Proposition \ref{propextremalgeodesics}). Notice that these two geodesics cannot intersect, and therefore they bound a compact region $K$ of the disc $\Omega$.
	
	If any of them is free boundary, because of property $(\star)$, we may use the positive function given in Lemma \ref{lemfirstjacobi} to construct an appropriate variation of that geodesic by curves with extremities in different components of $K\cap \partial \Omega$ that are, moreover, strictly shorter than that geodesic. Hence, the distance between the extremities of some curves in the variation is strictly smaller than $d(p,q)$ or $d(r,s)$. If any of them is not free boundary, combining the angle conditions satisfied by $\gamma_+$ and $\gamma_{-}$ at $\partial \Omega$ and the first variation formula, we can also produce variations of these curves in an appropriate direction so that the extremities of the varying curves lie in different components of $K\cap\partial \Omega$ and are, moreover, at distance strictly smaller than $d(p,q)$ or $d(r,s)$, because the curves in the variation are strictly shorter than $\gamma_{+}$ and $\gamma_{-}$.
	
	Hence, in any of the four possible cases for $\gamma_{+}$ and $\gamma_{-}$, after minimising distances between pair of points in different components of $K\cap \partial \Omega$, we find a local minimum $\{x,y\}\subset \partial \Omega$ of $\mathcal{D}$ that has extremities in different components of $int(K)\cap\partial \Omega$. By Proposition \ref{proplocalminima}, $\{x,y\}$ bounds a free boundary stable minimising geodesic. But this contradicts the assumption that $\Omega$ satisfies property $(\star)$.
	
	Thus, the configuration $p\rightarrow r\rightarrow s\rightarrow q$ leads to a contradiction. It follows that $r$ and $s$ must lie in different boundary arcs determined by $p$ and $q$, as we wanted to prove.
\end{proof}

	Before stating the next key propositions, it is convenient to introduce some terminology.
	
	Let $\{c_t\}$, $t\in I$, be a smooth family of smoothly embedded curves in the disc $\Omega$, whose extremities lie in $\partial \Omega$ at their extremities, and such that $I$ is a proper subinterval of $[0,1]$. (The interval may be closed, or not, in any of its extremities). The family $\{c_t\}$ is called \textit{strictly monotone} if there exists a point $x\in \partial \Omega$, that is not an extremity of any of the curves $c_t$, such that either of the following conditions hold:
	
	$(1)$ all curves $c_t$ are proper, and the closed regions $\Omega_t$ of the disc $\Omega$ determined by the curves $c_t$ that contain the given point $x$ are \textit{strictly nested}, \textit{i.e.} either $\Omega_t\subset int(\Omega_u)$ for all $t<u$ in $I$ or $\Omega_u\subset int(\Omega_t)$ for all $t<u$ in $I$.
	
	$(2)$ all curves $c_t$ are arcs of $\partial \Omega$, and the closed arcs that contain $x$ are \textit{strictly nested}, in the sense that $c_t\subset int(c_u)$ for all $t<u$ in $I$ or $c_u\subset int(c_t)$ for all $t<u$ in $I$.
	
	This definition does not depend on the choice of the point $x\in\partial \Omega$.
	
	Finally, a sweepout $\{p_t,q_t\}$ of $\partial \Omega$ is called \textit{regularly strictly monotone on an interval $I\subset [0,1]$} when there exists a strictly monotone smooth family of curves $\{c_t\}$, $t\in I$, all of them either properly embedded in $\Omega$ or else contained in $\partial \Omega$, such that the extremities of $c_t$ are the points $p_t$ and $q_t$ for every $t\in I$.
	
\begin{prop}\label{propstarcritical1}
	Let $(\Omega,g)$ be a totally convex Riemannian disc with smooth boundary in a complete Riemannian surface. Assume that $(\Omega,g)$ has property $(\star)$.
	
	Let $\{p,q\}\subset \partial \Omega$ be a critical point of $\mathcal{D}$ with $p\neq q$. If $p$ and $q$ are joined by a free boundary minimising geodesic $\gamma$, then there exists a sweepout $\{p_t,q_t\}$ of $\partial \Omega$ that has the following properties:
	\begin{itemize}
		\item[$i)$] $p_{1/2}=p$ and $q_{1/2}=q$.
		\item[$ii)$] $d(p_t,q_t)< d(p,q)$ for every $t\neq 1/2$.
		\item[$iii)$] For all $t$ in a small open interval containing $1/2$, the sweepout is regularly strictly monotone with $\{p_t,q_t\}=\partial c_t$ for a smooth family of smooth curves in $\Omega$ with $c_{1/2}=\gamma$.
		\item[$iv)$] In a small neighbourhood of $t=1/2$, $t\mapsto L(c_t)$ has a strictly negative second derivative at the critical point $t=1/2$.
	\end{itemize}	 
\end{prop}
\begin{proof}
	Choose a positive orthonormal frame $\{\gamma'(s),X(\gamma(s))\}$ along $\gamma$. Since $(\Omega,g)$ has property $(\star)$, we may extend the vector field given in Lemma \ref{lemfirstjacobi} to a vector field on $\Omega$ with compact support, that is tangent to $\partial \Omega$, and whose flow $\phi$ has the following properties: For every $t\in (-\delta,\delta)$ sufficiently small, the curves $c_{t+1/2}=\phi_{t}(\gamma)$ are a smooth family of smooth curves with extremities in $\partial \Omega$ that is strictly monotone, satisfy $c_{1/2}=\phi_0(\gamma)=\gamma$, and is such that there exists $\theta>0$ so that
	\begin{equation*}
		d(p_t,q_t)\leq L(c_t)\leq L(c_{1/2}) - \theta(t-1/2)^2 = d(p,q) - \theta(t-1/2)^2
	\end{equation*} 
	for all $t$ in a sufficiently small interval $(1/2-\delta,1/2+\delta)$, by virtue of the second variation formula \eqref{eqsecondvariation}. Notice in particular that the extremities $\{p_t,q_t\}$ of $c_t$ satisfy the properties $i)$, $iii)$ and $iv)$, but also $ii)$ for parameters close enough to $t=1/2$.
	
	Thus, in order to prove the proposition, it is enough to extend the smoothly monotone family $\{p_t,q_t\}$, $t\in (1/2-\delta,1/2+\delta)$, to a sweepout of $\partial \Omega$ in such way that $d(p_t,q_t)<d(p,q)$ for every $t\in [0,1]$. This will be achieved by iterating the Birkhoff chord shortening process.
	
	Recall that the Birkhoff chord shortening process describes how to shorten piecewise smooth curves with extremities in $\partial \Omega$, of uniformly bounded length, while keeping the extremities in $\partial \Omega$, in a continuous way. Roughly speaking, the process has two steps. In the first step, the process takes a (possibly immersed) piecewise smooth curve and sends it to a (possibly immersed) piecewise smooth curve that is made of minimising geodesics between its vertices, and that intersects $\partial \Omega$ orthogonally. In the second step, the process produces a continuous interpolation between the initial and the final curve. In both steps, the basic operation is either the substitution of a short piece of the curve by the minimising geodesic between the extremities of the piece, or, in case one of the extremities lies in $\partial \Omega$, it is the substitution of that short piece of the curve by the shortest path between the other vertex and $\partial \Omega$. Thus, along the process, the length of the curve does not increase. We refer the reader to \cite{GluZil}, \cite{Zho} and \cite{DonMon} for detailed descriptions. (\textit{Cf}. Section 2 of \cite{Cro}, where a concise description of the standard Birkhoff curve shortening process in the case of closed curves is given). 
	
	It is well-known that the iteration of the process either terminates at a point, or converges to a free boundary geodesic, which may be immersed. (See, for instance, the Lemma in page 71 in \cite{GluZil}). 
	
	In the situation we are analysing, we have control on where curves move during the process. In fact, denote by $\Omega_1$ and $\Omega_2$ the two closed regions of the disc $\Omega$ determined by $\gamma$, labelled so that $c_{1/2-\delta}$ belongs to $\Omega_1$ and $c_{1/2+\delta}$ belongs to $\Omega_2$. Since $\gamma$ is a free boundary minimising geodesic, both regions are locally convex in the sense that the minimising geodesic joining two points of $\Omega_i$, or joining a point of $\Omega_i$ that is sufficiently close to $\partial\Omega$ to the nearest point in $\partial \Omega$, still lie in $\Omega_i$. Thus, the Birkhoff chord shortening process starting on any of the curves $c_{\pm\delta+1/2}$ will remain in the respective region $\Omega_i$. \\
	
\noindent \textbf{Claim:} The iteration of the Birkhoff chord shortening process starting at both curves $c_{\pm \delta+1/2}$ terminates at points.  \\

	The Claim is proven by contradiction. Suppose it is not true. Then, the Birkhoff chord shortening process starting at $c_{\delta+1/2}$, say, produces an immersed free boundary geodesic $\alpha$ of $(\Omega,g)$. (The other case is analogous). By virtue of the previous remarks on the convexity of $\Omega_2$, the geodesic $\alpha$ actually lies in $\Omega_2$. Also, the extremities of $\alpha$ are disjoint from $\{p,q\}$, because otherwise the two free boundary geodesics $\alpha$ and $\gamma$ would coincide, which is absurd since $L(\alpha)\leq L(c_{1/2+\delta})<L(\gamma)$. Even more, since $\alpha$ lies in one side of $\gamma$, it cannot intersect $\gamma$ at any other points either, because otherwise $\alpha$ and $\gamma$ would have the same tangent at some point and so be the same geodesic.
	
	Thus, there exists a connected region $R\subset \Omega_2$ of the disc $\Omega$ such that $\partial R$ is the union of $\gamma$, pieces of $\alpha$ (or the even the whole $\alpha$ in case $\alpha$ is embedded) and the two arcs of $\partial \Omega$ bounded between the extremities of $\gamma$ and $\alpha$. In particular, $\partial R$ is a piecewise smooth curve, whose pieces are either geodesic ($\gamma$ and pieces of the immersed curve $\alpha$) or have non-negative geodesic curvature (the arcs of $\partial \Omega$), and whose vertices have angles either $\pi/2$ (if the vertex lies in $\partial \Omega$) or strictly smaller than $\pi$ from inside $R$ (if the vertex is a point of self-intersection of $\alpha$).
	
	Next, minimise the length of curves in $R$ with extremities in different components of $\partial R\cap \partial \Omega$. By the previous observations about the local convexity of the region $R$, the standard replacement by broken geodesics yields another minimising sequence which converges to a minimising smooth curve $\beta\subset R$, with positive length.
	
	Since the removal of immersed loops decreases length, it is clear that the minimising curve $\beta$ is embedded. Also, by virtue of the first variation formula of length, it is straighforward to check that $\beta$ is a geodesic that meets $\partial \Omega$ orthogonally. Since $L(\beta)\leq L(\alpha)<L(\gamma)$, arguing similarly as before we conclude that $\beta$ has no point in common with $\gamma$.
	
	To finish the contradiction argument, two cases must be considered. The first case is when the extremities of $\beta$ are also disjoint from the extremities of $\alpha$. In this case, we may use the second variation formula \eqref{eqsecondvariation} together with the minimisation property enjoyed by $\beta$ to conclude that $\beta$ is a free boundary stable geodesic with extremities in different arcs of $\partial R\setminus (\gamma \cup \alpha)$. But this scenario is ruled out by property $(\star)$.
	
	The other possibility is that $\beta$ has at least one extremity in common with $\alpha$. Since these two curves are free boundary geodesics, they must coincide. \textit{A posteriori}, we conclude that $\alpha$ must be embedded, and minimise length among curves that lie inside the region $R$ and have extremities in $\partial R\cap \partial \Omega$. By the second variation formula and Lemma \ref{lemfirstjacobi}, we conclude that $\alpha$ must be a free boundary stable geodesic. And this is again a contradiction with property $(\star)$.
	
	In conclusion, in all cases, the failure of the claim leads to the existence of a free boundary stable geodesic in $(\Omega,g)$, which contradicts property $(\star)$. The claim follows.
	
	In view of the shortening properties of the iterated Birkhoff chord shortening process, it follows immediately from our initial construction and the Claim that a sweepout with the desired properties $i) - iv)$ exist, as we wanted to prove.
\end{proof}

	The next proposition has a similar proof. Notice that it applies, in particular, to a pair of extremal minimising geodesics whose extremities are a non-trivial critical point of $\mathcal{D}$, as soon as none of them is free boundary.

\begin{prop}\label{propstarcritical2}
	Let $(\Omega,g)$ be a totally convex Riemannian disc with smooth boundary in a complete Riemannian surface. Assume that $(\Omega,g)$ has property $(\star)$. Denote by $T$ the unit tangent vector field to $\partial \Omega$ that is compatible with its orientation.
	
	Let $\{p,q\}\subset \partial \Omega$ be a critical point of $\mathcal{D}$ with $p\neq q$. Suppose $p$ and $q$ are joined by two minimising geodesics $\gamma_1$ and $\gamma_2$ such that 
		\begin{equation*} 
			\langle \nu_{\gamma_2}^{T}(p), T(p)\rangle \leq 0 \leq \langle \nu^T_{\gamma_1}(p), T(p)\rangle ,
		\end{equation*} 
		and
		\begin{equation*}
			\langle \nu_{\gamma_1}^{T}(q),T(q)\rangle \leq 0 \leq  \langle \nu^T_{\gamma_2}(q), T(q)\rangle,
		\end{equation*}
		Finally, suppose none of the two geodesics $\gamma_1$ and $\gamma_2$ is free boundary. 
		
		Then there exists a sweepout $\{p_t,q_t\}$ of $\partial\Omega$ that has the following properties:
	\begin{itemize}
		\item[$i)$] $p_{1/2}=p$ and $q_{1/2}=q$.
		\item[$ii)$] $d(p_t,q_t)< d(p,q)$ for every $t\neq 1/2$.
		\item[$iii)$] For all $t\leq 1/2$ sufficiently close to $1/2$, the sweepout is regularly strictly monotone with $\{p_t,q_t\}=\partial b_t$ for a family of curves with $b_{1/2}=\gamma_1$ and $\partial b_{t}$ belonging to the the arc of $\partial \Omega\setminus \{p,q\}$ into which $-T(p)$ and $T(q)$ points for $t<1/2$.
		\item[$iv)$] For all $t\geq 1/2$ sufficiently close to $1/2$, the sweepout is regularly strictly monotone with $\{p_t,q_t\}=\partial c_t$ for a family of curves with  $c_{1/2}=\gamma_2$ and $\partial c_{t}$ belonging to the the arc of $\partial \Omega\setminus \{p,q\}$ into which $T(p)$ and $-T(q)$ points for $t>1/2$.
		\item[$v)$] There exists a constant $\theta>0$ such that $d(p_t,q_t)\leq d(p,q)-\theta|t-1/2|$ near $t=1/2$.
	\end{itemize}	 
\end{prop}
\begin{proof}
	Let $\Omega_1$ be the closed component of $\Omega$ determined by $\gamma_1$ that does not contain $\gamma_2$, and let $\Omega_2$ be the closed component of $\Omega$ determined by $\gamma_2$ that does not contain $\gamma_1$.
	
	Since neither $\gamma_1$ nor $\gamma_{2}$ is free boundary, it follows from the assumptions that 
	\begin{equation}\label{eqauxz1}
		\langle -T(p),\nu^T_{\gamma_1}(p) \rangle + \langle T(q),\nu^T_{\gamma_1}(q) \rangle < 0
	\end{equation}
and
	\begin{equation}\label{eqauxz2}
		\langle T(p),\nu^T_{\gamma_2}(p) \rangle + \langle -T(q),\nu^T_{\gamma_2}(q) \rangle < 0.
	\end{equation}
	
	Notice that the vectors $-T(p)$ and $T(q)$ point towards $\Omega_1$. Extend these two vectors to a vector field $X$ along $\gamma_1$ in such way that $X$ points towards $\Omega_1$ at all points, and then extend this vector field to a vector field of $\Omega$ that is tangent to $\partial \Omega$ and vanishes outside a compact set. The flow of this vector field, $\phi$, generates curves $b_{t}=\phi_{1/2-t}(\gamma_1)$ with extremities $\{p_t,q_t\}$ in $\partial \Omega$ that are mutually disjoint and lie in $\Omega_1$, at least for all $t\in [1/2-\delta,1/2]$, for some $\delta>0$ sufficiently small. By construction, the family $\{b_t\}$ is strictly monotone. (Notice that if $\gamma_1$ is proper, then all $b_t$ are proper if $\delta$ is sufficiently small, whereas if $\gamma_1$ is an arc of $\partial \Omega$, then all $b_t$ are boundary arcs as well).
	
	By the first variation formula \eqref{eqfirstvariation} at the geodesic $\gamma_1=b_{1/2}$, and by \eqref{eqauxz1}, there exists some $\theta>0$ such that
\begin{equation*}
	d(p_t,q_t)\leq L(b_t) \leq L(\gamma_1)-\theta(1/2-t)=d(p,q)-\theta(1/2-t)
\end{equation*}
for all $t\in[1/2-\delta,1/2]$, possibly after decreasing $\delta$ if necessary. Thus, in particular, the extremities $\{p_t,q_t\}=\partial b_t$ satisfy properties $i)$, $ii)$, $iii)$ and $v)$ for all $t\in[1/2-\delta,1/2]$.

	Arguing similarly about $\gamma_2$, we can extend the vectors $T(p)$ and $-T(q)$, that point towards $\Omega_2$, to some vector field in $\Omega$ with compact support, tangent to $\partial \Omega$, and whose flow, $\psi$, will generate curves $c_{t}=\psi_{t-1/2}(\gamma_2)$ with extremities $\{p_t,q_t\}$ in $\partial \Omega$ that are mutually disjoint and lie in $\Omega_2$, at least if $t\geq 1/2$ is small enough. Arguing as before using the first variation formula \eqref{eqfirstvariation} and \eqref{eqauxz2}, we check that the strictly monotone family $\{p_t,q_t\}=\partial c_t$ satisfies properties $i)$, $ii)$, $iv)$ and $v)$ for all $t\in[1/2,1/2+\delta]$, after possibly decreasing $\delta$ and $\theta$.
	
	Gluing together these two families, we construct a continuous family $\{p_t,q_t\}\subset \partial \Omega$, $t\in [1/2-\delta,1/2+\delta]$, that satisfies properties $i)$, $iii)$, $iv)$ and $v)$, and thus also $ii)$ for parameters close enough to $t=1/2$.
	
	If one of the geodesics $\gamma_i$ is an arc of $\partial \Omega$, it is obvious how to continue the sweepout on the side of $\Omega_i$ without ever increasing the distance between the points of the sweepout: just shrink the arc inside $\partial \Omega \cap \Omega_i$ continuously to a point. Thus, in order to prove the proposition, it remains to consider the case where one of the extremal minimising geodesics, say $\gamma_2$, is not an arc of $\partial \Omega\cap \Omega_2$. (Up to relabelling, the argument is the same if the other geodesic is not a boundary arc).
	
	By the angle conditions between $\gamma_2$ and $\partial \Omega$, it is clear that $\Omega_2$ is locally convex in the sense that that the minimising geodesic joining two points of $\Omega_2$, or joining a point of $\Omega_2$ sufficiently close to $\partial \Omega$ to the nearest point in $\partial \Omega$, still lies in $\Omega_2$. Thus, the Birkhoff chord shortening process starting at the curve $c_{\delta+1/2}\subset \Omega_2$ will stay in the region $\Omega_2$. 

	If the process terminates at a point, then we have a family of boundary points to form (half of a) sweepout with the desired properties. If it does not, then the process terminates at an immersed free boundary geodesic $\alpha$ in $\Omega_2$. By assumption, $\gamma_2$ forms angles $\leq \pi/2$ with $\partial \Omega$ from within $\Omega_2$. Hence, $\alpha$ cannot have an extremity in common with $\gamma_2$, otherwise they would coincide, a contradiction with the assumption that $\gamma_2$ is not free boundary. Since $\alpha$ lies on one side of $\gamma_2$, by the same reason it cannot touch it at any interior point either.
	
	Thus, $\gamma_2$ and $\alpha$ determine a region $R$ whose boundary consists of arcs of geodesics or curves with non-negative geodesic curvature, which moreover meet at angles $\pi/2$ at $\partial \Omega$ and strictly less than $\pi$ from within $R$ at self-intersection points of $\alpha$. Minimise the length of curves in $R$ with extremities in different components of $R\cap \partial \Omega$. By the convexity properties of $R$, there exists a minimising curve $\beta\subset R$ in this class with positive length.	
	
	As argued in the proof of the Claim in Proposition \ref{propstarcritical1}, $\beta$ is an embedded free boundary geodesic that has no point in common with $\gamma_2$. If the extremities of $\beta$ are disjoint from the extremities of $\alpha$, then by the second variation formula \eqref{eqsecondvariation} and the minimisation property enjoyed by $\beta$ we conclude that $\beta$ is a free boundary stable geodesic with extremities in different arcs of $\partial R\setminus (\gamma_2 \cup \alpha)$, a contradiction with by property $(\star)$. The other possibility is that $\beta$ has at least one extremity in common with $\alpha$, in which case these two free boundary geodesics coincide. \textit{A posteriori}, we conclude that $\alpha$ is embedded, and has the least length among curves that lie inside the region $R$ and have extremities in different components $\partial R\cap \partial \Omega$. By the second variation formula \eqref{eqsecondvariation} and Lemma \ref{lemfirstjacobi}, $\alpha$ must be a free boundary stable geodesic, which is again a contradiction with property $(\star)$.
	
	Therefore, in any of the cases considered the family $\{p_t,q_t\}$ can indeed be extended, via the Birkhoff curve shortening process, to a sweepout of $\partial \Omega$ so that $t=1/2$ is the only maximum of $d(p_t,q_t)$, and all other properties described in $i)-v)$ are satisfied. This finishes the proof.
\end{proof}

\begin{prop} \label{propstarcriticalindex}
	Let $(\Omega,g)$ be a totally convex Riemannian disc with smooth boundary in a complete Riemannian surface. Let $\{p,q\}\subset \partial \Omega$ be a critical point of $\mathcal{D}$ with $d(p,q)>0$. Assume $\{p_t,q_t\}$ is a sweepout of $\partial \Omega$ that satisfies properties $i)-iv)$ described in Proposition \ref{propstarcritical1}.
	
	If $p$ and $q$ are joined by a free boundary minimising geodesic $\gamma$ of free boundary index $\geq 2$, then there exists a sweepout $\{p_t',q_t'\}$ of $\partial \Omega$ such that $d(p_t',q_t')< d(p,q)$ for every $t\in [0,1]$.
\end{prop}
\begin{proof}
	By the hypotheses, there exists a sweepout $\{p_t,q_t\}$ of $\partial \Omega$ that satisfies $i)$, $ii)$, $iii)$ and $iv)$ in Proposition \ref{propstarcritical1}. In particular, in small interval $(1/2-\delta,1/2+\delta)$, the sweepout $\{p_t,q_t\}$ is regularly strictly monotone, that is, $\{p_t,q_t\}=\partial c_t$ with $c_{1/2}=\gamma$ for some smoothly varying strictly monotone family of smooth curves $c_t$.

	By construction, and by the second variation formula \eqref{eqsecondvariation}, $L(c_t)$ is a function that has a critical point at $t=1/2$, and a strictly negative second derivative at that point, equal to $Q(uX,uX)<0$, for $u>0$ solving \eqref{eqfirstjacobieigen}, see Proposition \ref{propstarcritical1}. Since we assume the free boundary index of $\gamma$ is at least two, there exists $Y\in \Gamma(N\gamma')$ linearly independent with $uX$ such that $Q(uX,Y)=0$ and the quadratic form $Q$ is negative definite in the linear space they generate. Moreover, since $\gamma$ is stable for variations with fixed extremities, 
	\begin{equation*}
		(Y(p),Y(q)) \quad \text{and} \quad \left(\frac{d}{dt}_{|_{t=1/2}}p_t,\frac{d}{dt}_{|_{t=1/2}}q_t\right) = (u(0)X(0),u(a)X(a))
	\end{equation*}
are linearly independent vectors in $T_p\partial\Omega\times T_q\partial\Omega$. Extend $Y$ to a vector field in $\Omega$ with compact support that is tangent to $\partial \Omega$, and denote by $\psi_s$ the flow generated by $Y$.
	
	The curves $\psi_s\circ c_t$ lie within $\Omega$ and have extremities $\{p_t^s,q_t^s\}$ in $\Gamma$. By construction, $L(\psi_s(c_t))$ is a function of two variables $(t,s)\in (1/2-\delta,1/2+\delta)\times (-\infty,\infty)$ that is smooth, has a critical point at $(1/2,0)$ (by the first variation formula \eqref{eqfirstvariation}, since $c_0=\gamma$ is a free boundary geodesic), and its Hessian is negative definite at $(1/2,0)$ (by the second variation formula \eqref{eqsecondvariation}, property $iv)$ and our choice of vector field $Y$). Thus, $(t,s)=(1/2,0)$ is a unique maximum. Hence, possibly after deceasing $\delta$ if necessary, we can find a continuous function $s : t\in [0,1] \rightarrow [0,\eta)$, that is zero outside the interval $[1/2-\delta,1/2+\delta]$, so that 
	\begin{equation*}
		d(p_{t}^{s(t)},q_{t}^{s(t)})<d(p,q) \quad \text{for all} \quad t\in [1/2-\delta,1/2+\delta].
	\end{equation*}
	
	The family $\{p_t^{s(t)},q_t^{s(t)}\}$, which coincides with the original one for $t\notin [1/2-\delta,1/2+\delta]$ by construction, is then a sweepout of $\partial \Omega$ such that 
	\begin{equation*}
		\max_{t\in[0,1]} d(p_t^{s(t)},q_t^{s(t)}) < d(p,q),
	\end{equation*}  
	as we wanted to construct.
\end{proof}

\begin{prop}  \label{propstarcriticalsimultaneouslystationary}
	Let $(\Omega,g)$ be a totally convex Riemannian disc with smooth boundary in a complete Riemannian surface. Assume $(\Omega,g)$ has property $(\star)$. Denote by $T$ be the unit tangent vector field to $\partial \Omega$ that is compatible with its orientation.
	
	Let $\{p,q\}\subset \partial \Omega$ be a critical point of $\mathcal{D}$ with $p\neq q$. Suppose $p$ and $q$ are joined by two minimising geodesics $\gamma_1$ and $\gamma_2$, neither of them free boundary, such that 
		\begin{equation*} 
			\langle \nu_{\gamma_2}^{T}(p), T(p)\rangle \leq 0 \leq \langle \nu^T_{\gamma_1}(p), T(p)\rangle ,
		\end{equation*} 
		and
		\begin{equation*}
			\langle \nu_{\gamma_1}^{T}(q),T(q)\rangle \leq 0 \leq  \langle \nu^T_{\gamma_2}(q), T(q)\rangle.
		\end{equation*}
	
	Assume $\{p_t,q_t\}$ is a sweepout of $\partial \Omega$ that satisfies properties $i)-v)$ described in Proposition \ref{propstarcritical2} with respect to the geodesics $\gamma_1$ and $\gamma_2$.
	
	If $\gamma_1$ and $\gamma_2$ are not simultaneously stationary, then there exists a sweepout $\{p_t',q_t'\}$ of $\partial \Omega$ such that $d(p_t',q_t')< d(p,q)$ for every $t\in [0,1]$.
\end{prop}
\begin{proof}
	By the hypotheses, there exists a sweepout $\{p_t,q_t\}$ of $\partial \Omega$ that satisfies $i)$, $ii)$, $iii)$ and $iv)$ in Proposition \ref{propstarcritical2}. In particular, in small intervals $(1/2-\delta,1/2]$ and $[1/2,1/2+\delta)$, the sweepout $\{p_t,q_t\}$ is regularly strictly monotone, where the family of curves $b_t$ with $\partial b_t=\{p_t,q_t\}$ ends at $\gamma_1$ in the first interval, and the family of curves $c_t$ with $\partial c_t=\{p_t,q_t\}$ starts at $\gamma_2$ in the second interval.
	
	Suppose, by contradiction, that $\gamma_1$ and $\gamma_2$ are not simultaneously stationary. Then, there exists $v_1\in T_p\partial \Omega$ and $v_2\in T_q\partial \Omega$ such that
	\begin{equation}\label{eqauxx1}
		\langle v_1,\nu^T_{\gamma_1}(p) \rangle + \langle v_2,\nu^T_{\gamma_1}(q) \rangle < 0
	\end{equation}
and
	\begin{equation}\label{eqauxx2}
		\langle v_1,\nu^T_{\gamma_2}(p) \rangle + \langle v_2,\nu^T_{\gamma_2}(q) \rangle < 0.
	\end{equation}
	Extend $v_1$ and $v_2$ to tangent vector field $X$ on $\Omega$, that is tangent to $\partial \Omega$ on boundary points and whose support intersected with $\partial \Omega$ is contained in the union of the arcs of $\partial \Omega$ covered by the points $p_t$ and $q_t$ for every $t\in [1/2-\delta,1/2+\delta]$. (There is no ambiguity because $p\neq q$).
	
	Let $\phi_s$ be the flow of $X$ on $\Omega$, and write $\phi_s(\{p_t,q_t\})=\{p^s_t,q_t^s\}$ for all $t\in [0,1]$ and all $s\in [0,+\infty)$. Notice that $\{p_t^s,q_t^s\}=\{p_t,q_t\}$ if $t\notin [1/2-\delta,1/2+\delta]$.
	
	The function $L(\phi_{s}(c_t))$ is a smooth function in the variables $(t,s)\in [1/2,1]\times [0,+\infty)$ near $(t,s)=(1/2,0)$. It follows immediately from \eqref{eqauxx2} and the first variation formula \eqref{eqfirstvariation} that the curve $c_{1/2}=\gamma_2$ varies by curves $\phi_s(c_{1/2})$ in such way that the first derivative of the length in the $s$ variable is strictly negative at $s=0$. By continuity, the same assertion about the variation of $c_t$ by the  curves $\phi_s(c_{t})$ is true for all $t\geq 1/2$ sufficiently close to $t=1/2$. Hence, for some small $0<\delta'<\delta$ and some small $\eta>0$, 
	\begin{equation*}
		d(p^s_t,q_t^s)\leq L(\phi_s(c_t))<L(c_t)\leq d(p,q)
	\end{equation*}
	for all $t\in[1/2,1/2+\delta']$ and $0<s<\eta$.
	
	By a similar reasoning about the function $L(\phi_{s}(b_t))$ in the variables $(t,s)\in[0,1/2]\times [0+\infty)$, which is smooth near $(t,s)=(1/2,0)$, the curve $b_{1/2}=\gamma_1$ varies by curves $\phi_s(b_{1/2})$ in such way that the first derivative of the length in the $s$ variable at $s=0$ is strictly negative, as a consequence of \eqref{eqauxx1} and the first variation formula \eqref{eqfirstvariation}. By continuity, the same is true about the variation of $b_t$ by the curves $\phi_s(b_{t})$ for all $t\geq 1/2$ sufficiently close to $t=1/2$. Thus, decreasing $\delta'$ and $\eta$ if necessary, we guarantee that the inequalities 
	\begin{equation*}
		d(p^s_t,q_t^s)\leq L(\phi_s(b_t))<L(b_t)\leq d(p,q)
	\end{equation*}
	holds as well for all $t\in[1/2-\delta',1/2]$ and $0<s<\eta$.
	
	Therefore the continuous map 
	\begin{equation*}
		(t,s)\in [1/2-\delta',1/2+\delta']\times [0,\eta]\mapsto d(p_t^s,q_t^s)\in [0,d(p,q)]
	\end{equation*}
	is such that $(1/2,0)$ is the unique maximum $d(p,q)$. Now, it is possible to chose some continuous function $s : t\in [0,1] \rightarrow [0,\eta)$, that is zero outside the interval $[1/2-\delta',1/2+\delta']$, in such way that 
	\begin{equation*}
		d(p_{t}^{s(t)},q_{t}^{s(t)})<d(p,q) \quad \text{for all} \quad t\in [1/2-\delta',1/2+\delta'].
	\end{equation*}
	
	The family $\{p_t^{s(t)},q_t^{s(t)}\}$, which coincides with the original one for $t\notin [1/2-\delta',1/2+\delta']$, is therefore a sweepout of $\partial \Omega$ such that 
	\begin{equation*}
		\max_{t\in[0,1]} d(p_t^{s(t)},q_t^{s(t)}) < d(p,q).
	\end{equation*}
\end{proof}

	We are now ready for the proof of Theorem B.

\begin{thm} \label{thmBbis}
	Let $(\Omega,g)$ be a totally convex Riemannian disc with smooth boundary in some Riemannian manifold. Assume $(\Omega,g)$ has property $(\star)$. Then every critical point $\{x,y\}\subset \partial \Omega$ of $\mathcal{D}$ with $x\neq y$ is such that
	\begin{equation*}
		d(x,y)\geq \mathcal{S}(\partial \Omega).
	\end{equation*} 
	Moreover, let $\{p,q\}$ be a critical point with $d(p,q)= \mathcal{S}(\partial\Omega)$. Then,
	\begin{itemize}
		\item[$i)$] if there exists a free boundary minimising geodesic joining $p$ and $q$, then this geodesic has free boundary index one.
		\item[$ii)$] if two minimising geodesics $\gamma_1$ and $\gamma_2$ joining the points $p$ and $q$ satisfy $\langle \nu_{\gamma_1}^{T}(p), \nu^T_{\gamma_2}(p)\rangle \leq 0$ and $\langle \nu_{\gamma_1}^{T}(q), \nu^T_{\gamma_2}(q)\rangle \leq 0$, and none of them is free boundary, then $\gamma_1$ and $\gamma_2$ are simultaneously stationary.
	\end{itemize}
\end{thm}

\begin{proof}
	By Propositions \ref{propstarcritical1} and \ref{propstarcritical2}, every critical point $\{p,q\}\subset \partial \Omega$ of $\mathcal{D}$ with $p\neq q$ is such that $d(p,q)=\max_{t\in [0,1]}d(p_t,q_t)$ for some sweepout $\{p_t,q_t\}$ of $\partial \Omega$ with $p_{1/2}=p$ and $q_{1/2}=q$. By definition of $\mathcal{S}(\partial \Omega)$, the inequality $\mathcal{S}(\partial \Omega)\leq d(p,q)$ follows immediately.
	
	If $\{p,q\}\subset \partial \Omega$ is a critical point with $d(p,q)= \mathcal{S}(\partial \Omega)>0$, then there are two possibilities. Either $p$ and $q$ are joined by a free boundary geodesic, or there exists at least two minimising geodesics joining $p$ and $q$, none of them free boundary, that satisfy $\langle \nu_{\gamma_1}^{T}(p), \nu^T_{\gamma_2}(p)\rangle \leq 0$ and $\langle \nu_{\gamma_1}^{T}(q), \nu^T_{\gamma_2}(q)\rangle \leq 0$. (Proposition \ref{propextremalgeodesics}). In the first case, the index must be one, otherwise using Proposition \ref{propstarcriticalindex} we would be able to construct a sweepout that violates the definition of $\mathcal{S}(\partial \Omega)$. In the second case, Proposition \ref{propstarcriticalsimultaneouslystationary} guarantees that any such pair of geodesics $\gamma_1$ and $\gamma_2$ are simultaneously stationary, for the same reason. This finishes the proof.
\end{proof}

\begin{rmk}
	The attentive reader will have noticed that the arguments in this section proves Theorem \ref{thmBbis} under the slightly weaker assumption that there are no free boundary stable geodesic with length $\leq \mathcal{S}(\partial \Omega)$. We leave the details to the interested reader.
\end{rmk}

\section{Width, diameter and length}\label{sectionzoll}

	The following two general topological facts about continuous maps of the circle without fixed points will be used several times in what follows. Recall that a homeomorphism $\phi: S^1\rightarrow S^1$ is called \textit{monotone} when $\phi(x)$ turns around $S^1$ in the same direction as $x$ does.

\begin{lemm}\label{lemvarredura}
Let $\phi : S^1\rightarrow S^1$ be a continuous map without fixed points. Then, every sweepout $\{p_t,q_t\}$ of $S^1$ contains a pair $\{p_{t_0},\phi(p_{t_0})\}$.
\end{lemm}

\begin{proof} 
Let $\theta_0$ be the unique element of $[0,2\pi)$ such that $p_0=q_0=e^{i\theta_0}$. Let $\theta$, $\psi : [0,1]\rightarrow \mathbb{R}$ be the unique continuous angular functions determined so that $\theta(0)=\psi(0)=\theta_0$ and $p_t=e^{i\theta(t)}$, $q_t=e^{i\psi(t)}$ for every $t\in [0,1]$. Since $p_1=q_1$, we have $\psi(1)-\theta(1)\in 2\pi\mathbb{Z}$. 

We claim that $|\psi(1)-\theta(1)|\geq 2\pi$. In fact, otherwise $\psi(1)=\theta(1)$. Then the two maps $\theta$, $\phi: [0,1]\rightarrow \mathbb{R}$ are homotopic with fixed endpoints, by a continuous homotopy $\tilde{H} : [0,1]\times [0,1]\rightarrow \mathbb{R}$ such that $\tilde{H}(t,0)=\theta(t)$, $\tilde{H}(t,1)=\psi(t)$, $\tilde{H}(0,s)=\theta(0)=\psi(0)$ and $\tilde{H}(1,s)=\theta(1)=\psi(1)$ for all $t$ and $s$. The map $H :[0,1]\times [0,1]\rightarrow \mathcal{P}_*$ given by $H(t,s)=\{e^{iH(t,s)},e^{i\psi(t)}\}$ is a continuous homotopy between the sweepout $t\mapsto \{e^{iH(t,0)},e^{i\psi(t)}\}=\{e^{i\theta(t)},e^{i\psi(t)}\}=\{p_t,q_t\}$ and the constant map $t\mapsto \{e^{iH(t,1)},e^{i\psi(t)}\}=\{e^{i\psi(t)},e^{i\psi(t)}\}$ in $\mathcal{P}_*$, and this is a contradiction.

To finish the proof, we analyse the case $\psi(1)\geq \theta(1)+2\pi$. (The case $\psi(1)\leq \theta(1)-2\pi$ is treated similarly).  Since $\phi$ has no fixed points, $\phi(p_0)\neq p_0$ and there exists a unique $\alpha_0\in (\theta(0),\theta(0)+2\pi)$ such that $\phi(p_0)=e^{i\alpha_0}$. Let $\alpha : [0,1]\rightarrow\mathbb{R}$ be the unique angular function such that $\alpha(0)=\alpha_0$ and $\phi(p_t)=e^{i\alpha(t)}$ for all $t\in [0,1]$. 

Since $\phi$ has no fixed points, $\alpha(t)\in (\theta(t),\theta(t)+2\pi)$ for all $t\in [0,1]$. In fact, otherwise there would be, by continuity, some $t\in (0,1]$ such that either $\alpha(t)=\theta(t)$ or $\alpha(t)=\theta(t)+2\pi$. But then $\phi(p_t)=e^{i\alpha(t)}=e^{i\theta(t)}=p_t$, which contradicts the assumption that $\phi$ has no fixed points. 

	Since $\psi(0)=\theta(0)<\alpha(0)$ and $\psi(1)\geq \theta(1)+2\pi>\alpha(1)$, by continuity we find $t_0\in (0,1)$ such that $\psi(t_0)=\alpha(t_0)$. Therefore $\phi(p_{t_0}) = e^{i\alpha(t_0)}=e^{i\psi(t_0)}=q_{t_0}$. The proposition follows.
\end{proof}

\begin{lemm}\label{lemcontinuousinvolution}
	Let $\phi : S^1\rightarrow S^1$ be a continuous involution without fixed points. Then $\phi$ is a monotone homeomorphism, and $\phi$ induces a two-to-one cover $p \in S^1 \mapsto \{p,\phi(p)\} \in \mathcal{P}_*$ of a homotopically non-trivial loop in $\mathcal{P}_*$.
\end{lemm}
\begin{proof}
	By assumption, $\phi$ is continuous and satisfies $\phi\circ\phi=id_{S^1}$. In particular, $\phi: S^1\rightarrow S^1$ is a homeomorphism.
	
	Suppose, by contradiction, that $\phi$ is not monotone. Then there exists $\{p,\phi(p)\}$ such that one of the arcs $C$ of $S^1$ determined by it contains a some $\{q,\phi(q)\}$ in its interior.
	
	Let $\{\overline{q},\phi(\overline{q})\}$ be the subset in $C$ such that the closed arc $C_{\overline{q}}$ with $\partial C_{\overline{q}}=\{\overline{q},\phi(\overline{q})\}$ contained in $C$ has the least possible length. Since $\phi$ has no fixed points, possibly after interchanging $\overline{q}$ and $\phi(\overline{q})$ we may assume that $p$, $\overline{q}$, $\phi(\overline{q})$ and $\phi(p)$ appear in that order as we traverse $S^1$ in the counterclockwise direction. Moreover, $L(C_{\overline{q}})>0$ and $\phi(x)\notin C_{\overline{q}}$ for every $x\in C_{\overline{q}}$.
	
	If $q'\in C_{\overline{q}}$ is sufficiently close to $\overline{q}$, then $\phi(q')$ is sufficiently close to $\phi(\overline{q})$, and lies outside $C_{\overline{q}}$. Hence, in this case $\phi(q')$ must lie in the arc inside $C$ joining $\phi(\overline{q})$ to $\phi(p)$. Since $C_{\overline{q}}$ is connected, and $\phi(C_{\overline{q}})$ is therefore a connected set that contains $\phi(\overline{q})$, $\phi(q')$ and $\overline{q}=\phi(\phi(\overline{q}))$, it follows that $\phi(C_{\overline{q}})\supset\Gamma\setminus C_{\overline{q}}$. In particular, there exists $q''\in C_{\overline{q}}$ such that $\phi(q'')=\phi(p)$. Applying the involution $\phi$, we conclude that $q''=p$, a contradiction.
	
	Monotonicity implies that $\phi$ interchanges the two arcs of $S^1$ bounded by a pair $\{x,\phi(x)\}\in \Gamma$ as $x$ moves towards $\phi(x)$. Thus, this path lifts to an open continuous path in the oriented double cover of $\mathcal{P}_*$. In other words, this path is homotopically non-trivial in $\mathcal{P}_*$. The proposition follows. 
\end{proof}

\begin{prop}\label{propthmCvolta}
	Let $\Gamma$ be a smoothly embedded circle in a complete Riemannian manifold. Suppose that there exists a continuous map $p\in \Gamma \mapsto \phi(p)\in \Gamma$ such that $d(p,\phi(p))=diam(\Gamma)$ for every $p\in \Gamma$. Then $\mathcal{S}(\Gamma)=diam(\Gamma)$.
\end{prop}
\begin{proof}
	Since $d(p,\phi(p))=diam(\Gamma)>0$ for every $p\in \Gamma$, the continuous map $\phi$ has no fixed points. Therefore any sweepout $\{p_t,q_t\}$ of $\Gamma$ contains a pair $\{p_{t_0},\phi(p_{t_0})\}$, by Lemma \ref{lemvarredura}. Hence, for any sweepout of $\Gamma$, we have $\max_{t\in[0,1]}d(p_t,q_t))\geq d(p_{t_0},\phi(p_{t_0})))= diam(\Gamma)$. By definition of $\mathcal{S}(\Gamma)$, and the fact that this number is bounded from above by $diam(\Gamma)$, the proposition follows.
\end{proof}

	Moving towards the converse statement, we prove a series of propositions and lemmas. First, we prove a weak converse of Proposition \ref{propthmCvolta}.

\begin{lemm}\label{lemthmCidafracaA}
	Let $\Gamma$ be a smoothly embedded circle in a complete Riemannian surface. If $\mathcal{S}(\Gamma)=diam(\Gamma)$, then for every $p\in \Gamma$ there exists $q\in \Gamma$ such that $d(p,q)=diam(\Gamma)$. 
\end{lemm}
\begin{proof}
	Pick a point $p\in \Gamma$, and let $t\in [0,1]\mapsto p_t,q_t\in \Gamma$ be continuous maps such that $p_t=p$ for every $t\in [0,1]$, and $q_t$ traverses $\Gamma$ in the counter-clockwise direction while moving from $q_0=p$ to $q_1=p$. Then $\{p_t,q_t\}_{t\in [0,1]}$ is a sweepout of $\Gamma$. By continuity, the maximum distance between points $\{p_t,q_t\}$ is attained at some $t_0\in[0,1]$. Under the assumption that $\mathcal{S}(\Gamma)=diam(\Gamma)$, we have therefore $diam(\Gamma)=\mathcal{S}(\Gamma)\leq d(p_{t_0,}q_{t_0})\leq diam(\Gamma)$. Hence, $\{p=p_{t_0},q_{t_0}\}\subset \Gamma$ is a critical point of $\mathcal{D}$ with $d(p,q_{t_0})=diam(\Gamma)$.
\end{proof}

	Using Proposition \ref{propuniqu}, the previous lemma can be improved.

\begin{prop}\label{propthmCidafracaB}
	Let $\Gamma$ be a smoothly embedded circle in a complete Riemannian manifold that is the boundary of smoothly embedded totally convex disc. If $\mathcal{S}(\Gamma)=diam(\Gamma)$, then, for every point $p\in \Gamma$, there exists a unique $\phi(p)\in \Gamma$ such that $d(p,\phi(p))=diam(\Gamma)$. Moreover, the map $p\in \Gamma\mapsto \phi(p)\in \Gamma$ is a monotone homeomorphism.
\end{prop}
\begin{proof}
	By Lemma \ref{lemthmCidafracaA}, for any given $p \in \Gamma$, the set of points at distance $diam(\Gamma)$ from $p$ is non-empty. It is also clearly compact. Thus, taking into account the orientation of $\Gamma$, there are uniquely defined points $d_{-}(p)$ and $d_{+}(p)$ in $\Gamma$ at distance $diam(\Gamma)$ from $p$ such that $d_{+}(p)$ is closest to $p$ as we turn around $\Gamma$ in the counter-clockwise direction, and $d_{-}(p)$ is furthest from $p$ as we turn around $\Gamma$ in the counter-clockwise direction.
	
	Since the pairs $\{p,d_{+}(p)\}$ and $\{p,d_{-}(p)\}$ are both critical points of $\mathcal{D}$ and $p\neq d_{-}(p)$ and $p\neq d_{+}(p)$, the equality $d_{+}(p)=d_{-}(p)$ is an immediate consequence of Proposition \ref{propuniqu}. Thus, a map $\phi:\Gamma \rightarrow \Gamma$ with $d(p,\phi(p))=diam(\Gamma)$ for all $p\in \Gamma$ is well-defined. Notice also that this map satisfies $\phi(\phi(p))=p$ for all $p\in \Gamma$, for the same reason.
	
	To check the continuity of $\phi$, let $\{p_i\}\subset \Gamma$ be a sequence converging to a point $p\in \Gamma$. Every subsequence of the sequence $\{\phi(p_i)\}\subset \Gamma$ has a subsequence converging to some $q\in \Gamma$. Since $d(p,q)=\lim d(p_i,\phi(p_i))=diam(\Gamma)$, we have $q=\phi(p)$, by uniqueness of the point at distance $diam(\Gamma)$ from $p$ established earlier in this proof. Hence, $\{\phi(p_i)\}$ converges to $\phi(p)$. Since $p$ is arbitrary, $\phi$ is continuous at every point of $\Gamma$.
	
	The last assertion of the Proposition follows now immediately from Lemma \ref{lemcontinuousinvolution}.
\end{proof}

	Combining the previous results, we prove Theorem C.

\begin{thm}\label{thmCbis}
	Let $\Gamma$ be a smoothly embedded circle in a complete Riemannian surface.
	\begin{itemize}
		\item[$(a)$] If there exists a continuous map $\phi : \Gamma \rightarrow \Gamma$ such $d(x,\phi(x))=diam(\Gamma)$, then $\mathcal{S}(\Gamma)=diam(\Gamma)$.
		\item[$(b)$] If $\mathcal{S}(\Gamma)=diam(\Gamma)$, then for every $x\in \Gamma$ there exists $y\in \Gamma$ such that $d(x,y)=diam(\Gamma)$.
	\end{itemize}
	Assume, moreover, that $\Gamma$ is the boundary of a smoothly embedded totally convex Riemannian disc. Then $\mathcal{S}(\Gamma)=diam(\Gamma)$ if and only if there exists a continuous map $\phi : \Gamma \rightarrow \Gamma$ map such $d(x,\phi(x))=diam(\Gamma)$.
\end{thm}

\begin{proof}
	Assertion $a)$ was proven in Proposition \ref{propthmCvolta}. Assertion $b)$ follows from Proposition \ref{lemthmCidafracaA}. The final equivalence follows from Proposition \ref{propthmCidafracaB}.
\end{proof}

	The next elementary lemma characterises the case where the extrinsic and intrinsic diameters of $\Gamma$ are the same.

\begin{lemm}\label{lemdiamlength}
	Let $\Gamma$ be a smoothly embedded circle in a complete Riemannian manifold. The following assertions are equivalent:
	\begin{itemize}
		\item[$i)$] $diam(\Gamma)=L(\Gamma)/2$.
		\item[$ii)$] $\Gamma$ is a geodesic formed by two minimising geodesics of length $L(\Gamma)/2$.
	\end{itemize}
\end{lemm}
\begin{proof}
	Assume $i)$. Let $p$, $q\in \Gamma$ be such that $d(p,q)=diam(\Gamma)$. Notice that $\Gamma$ is the union of two arcs $\gamma_1$, $\gamma_2$ with extremities $p$ and $q$, which we label so that $L(\gamma_1)\leq L(\gamma_2)$. Since $diam(\Gamma)=d(p,q)\leq L(\gamma_1)\leq L(\Gamma)/2$, it follows from $i)$ that $\gamma_1$ must be a minimising geodesic joining $p$ and $q$ of length $L(\Gamma)/2$. Then, $\gamma_2$ also has length $L(\Gamma)/2=L(\Gamma)-L(\gamma_1)$, and the same argument shows that $\gamma_2$ is also a minimising geodesic joining the same pair of points. This proves $ii)$.
	
	Assume $ii)$. Since the two geodesics forming $\Gamma$ are minimising and have length $L(\Gamma)/2$, their extremities $p$ and $q$ satisfy $d(p,q)=L(\Gamma)/2$. Since $L(\Gamma)/2$ is an upper bound for the distance of any pair of points of $\Gamma$, $diam(\Gamma)=d(p,q)=L(\Gamma)/2$, which is $i)$.
\end{proof}

	Finally, we analyse the equality $\mathcal{S}(\Gamma)=L(\Gamma)/2$. (Theorem D).

\begin{thm}\label{thmDbis}
	Let $\Gamma$ be a smoothly embedded circle in a complete Riemannian manifold. The following assertions are equivalent:
	\begin{itemize}
		\item[$i)$] $\mathcal{S}(\Gamma)=L(\Gamma)/2$.
		\item[$ii)$] For every $x$, $y\in \Gamma$, the distance between $x$ and $y$ equals the length of the shortest arc of $\Gamma$ bounded by these two points.
	\end{itemize}
\end{thm}
\begin{proof}
	We begin by proving that $ii)$ implies $i)$. Every sweepout $\{p_t,q_t\}$ of $\Gamma$ satisfies 
	\begin{equation*}
		\max_{t\in [0,1]} d_{(\Gamma,g_{|_{\Gamma}})}(p_t,q_t) = L(\Gamma)/2, 
 	\end{equation*}
 	and the maximum is attained. The equality $d(x,y)=d_{(\Gamma,g_{|_{\Gamma}})}(x,y)$ for all $\{x,y\}\subset \Gamma$ thus implies, by the definition of $\mathcal{S}(\Gamma)$, that $\mathcal{S}(\Gamma)=L(\Gamma)/2$.
 	
 	Conversely, suppose $i)$. In particular, $diam(\Gamma)=\mathcal{S}(\Gamma)$. By Lemma \ref{lemthmCidafracaA}, given any $x\in \Gamma$ there exists $\phi(x)\in \Gamma$ such that $d(x,\phi(x)) = diam(\Gamma)= \mathcal{S}(\Gamma)=L(\Gamma)/2$. This point $\phi(x)$ is the unique point of $\Gamma$ with this property, because the two arcs determined by $x$ and $\phi(x)$ must then have the same length. Also, for any pair of points $x$, $y\in \Gamma$, the shortest arc of $\Gamma$ joining $x$ and $y$ is minimising as well, since it is part of one of the arcs bounded by $x$ and $\phi(x)$, which are both minimising because $d(x,\phi(x)) = L(\Gamma)/2$. Thus, $d(x,y)=d_{(\Gamma,g_{|_{\Gamma}})}(x,y)$, as we wanted to prove.  
\end{proof}   

\begin{rmk}\label{rmkjacobi} As a consequence of Theorem \ref{thmDbis}, a smoothly embedded circle whose width equals half its length is a closed embedded geodesic that has the following property as well: any of its periodic Jacobi fields must either never vanish or vanish exactly at a pair $\{x,y\}\subset \Gamma$ bounding two arcs of $\Gamma$ of the same length. In fact, if not there would be two conjugate points strictly between some minimising geodesic arc between two points $\{p,q\}$ such that $d(p,q)=L(\Gamma)/2$, a contradiction. Therefore the Morse index of this closed embedded geodesic is either zero or one. Both cases are actually possible, see Examples \ref{examhemisphere} and \ref{examcylhem} in Section 9.
\end{rmk}

	Collecting several the results proven so far, we prove Theorem E. 
\begin{proof}[Proof of Theorem E]
	Let $\Gamma$ be a smoothly embedded circle in a complete Riemannian manifold. According to Corollary \ref{cormaximumiscritical} and Theorem A (Theorem \ref{thmAbis}), if $0<\mathcal{S}(\Gamma)<diam(\Gamma)$, then there are at least two critical points of $\mathcal{D}$ in $\mathcal{P}$, one attaining the width and the other attaining the diameter of $\Gamma$.

	If, however, $\mathcal{S}(\Gamma)=diam(\Gamma)$, then Proposition \ref{propthmCidafracaB} implies that every $x\in \Gamma$ belongs to a pair $\{x,y\}\subset \Gamma$ such that $d(x,y)=diam(\Gamma)$. In particular, there are infinitely many non-trivial critical points of $\mathcal{D}$ in $\Gamma$.

	If, moreover, $\Gamma$ is the boundary of a smoothly embedded totally convex disc, more can be said. Proposition \ref{propthmCidafracaB} guarantees that there exists a monotone homeomorphism $\phi: \Gamma\rightarrow \Gamma$ such that $\{x,\phi(x)\}$ is a critical point of $\mathcal{D}$ with $d(x,\phi(x))=diam(\Gamma)$ for every $x\in \Gamma$. By the uniqueness property of critical points of $\mathcal{D}$ on the boundary of totally convex discs (Proposition \ref{propuniqu}), there are no other critical points of $\mathcal{D}$ besides the pairs $\{x,\phi(x)\}$, $x\in \Gamma$. Viewed in $\mathcal{P}_*$, the set of pairs $\{x,\phi(x)\}$ is an embedded circle that represents a homotopically non-trivial loop in $\mathcal{P}_*$, see Lemma \ref{lemcontinuousinvolution}. Theorem E is therefore proven.
\end{proof}

	To finish this section, we explain a construction of independent interest, which shows that no Riemannian disc with strictly convex boundary can not be a local maximum of $\mathcal{S}(\partial M)/L(\partial M)$, among Riemannian discs with strictly convex boundary. 

\begin{prop}\label{propnocriticalSL}
	Let $(M^2,g)$ be a compact Riemannian surface with connected, strictly convex boundary. Then, there exists a smooth one parameter family of conformal Riemannian metrics $g_s=\exp(f_s)g$, $s\in[0,\epsilon)$, and a positive constant $A>0$ such that $g_0=g$ and
	\begin{equation*}
		\frac{\mathcal{S}(\partial M,g_s)}{L(\partial M,g_s)} \geq \frac{\mathcal{S}(M,g)}{L(\Gamma,g)}+As+o(s) \quad \text{as $s$ goes to zero}.
	\end{equation*}
\end{prop}
\begin{proof}
	Denote by $V_\delta$ the collar neighbourhood of radius $\delta>0$ of $\partial M$ in $(M^2,g)$, where $\delta$ is a sufficiently small constant that will be specified later. Fix some non-negative number $s$, and let $\phi_s\in C^{\infty}(M)$, $s\geq 0$, be a smooth one-parameter family  with $\phi_0\equiv 1$ such that, for all $s\geq 0$,
\begin{itemize}
	\item[$i)$] $\phi_s(x)=1$ if $d(x,\partial M)\leq \delta/3$.
	\item[$ii)$] $\phi_s(x)=(1+s)^2$ if $d(x,\partial M) \geq 2\delta/3$
	\item[$iii)$] $\phi_s(x)\in [1,(1+s)^2]$ for all $x\in M$
\end{itemize}
	The metrics we are looking for will be of the form $g_s=\phi_s g$. By virtue of $i)$, all metrics $g_s$ coincide in $V_\delta$. In particular, $\partial M$ is strictly convex in $(M^2,g_s)$ and $L(\partial M,g_s)=L(\partial M,g)$ for every $s\geq 0$. Thus, it will be enough to show that $\mathcal{S}(\partial M,g_s)$ increases at least linearly near $s=0$.

	We begin with some simple estimates relating distances measured according to $g_s$ and $g_0=g$. Let $c:[0,1]\rightarrow M$ be any curve joining points $x$, $y\in M$. By property $iii)$, the length of $c$ with respect to the conformal metrics $g_s$ and $g$ can be compared as follows:
\begin{equation*}
	L(c,g) \leq L(c,g_s) \leq (1+s) L(c,g).
\end{equation*}  
Since $c$ is arbitrary,
\begin{equation*}
	d_{g}(x,y) \leq d_{g_s}(x,y) \leq (1+s) d_{g}(x,y).
\end{equation*}  
 In particular, if $\gamma$ is a minimising geodesic of $(M^2,g)$ joining points $x$, $y\in M$, 
\begin{equation} \label{eqauxSL1}
	d_{g_s}(x,y)\geq d_g(x,y)=L(\gamma,g)\geq \frac{1}{1+s}L(\gamma,g_s).
\end{equation}
	
	The first claim will allow us to specify a convenient $\delta>0$. \\
	
\noindent \textbf{Claim 1}: for every $\varepsilon>0$, there exists $\delta>0$ with the following property: For every minimising geodesic $\gamma:[0,a]\rightarrow M$ of $(M^2,g)$ such that $\gamma(0)\in \partial M$ and $L(\gamma,g)\geq \varepsilon$, there exists $t\in (0,a]$ such that $d(\gamma(t),\partial M)\geq \delta$.  \\

	In order to prove the Claim, we argue by contradiction: if not, there exists some $\varepsilon>0$ and a sequence of minimising geodesics $\gamma_i:[0,a_i]\rightarrow M$ of $(M^2,g)$, with $\gamma_i(0)\in \partial M$ and $L(\gamma_i)\geq \varepsilon$, such that they lie entirely in the collar neighbourhood $V_{1/i}$. Passing to the limit, these geodesics converge to an arc of $\partial M$ of length $\geq \varepsilon>0$ that is geodesic as a limit of geodesics. But this contradicts the assumption that $\partial M$ is strictly convex in $(M^2,g)$. \\
	
	In view of Claim 1, we define $\varepsilon=\mathcal{S}(\partial M,g)/10$ and choose the corresponding $\delta=\delta(\varepsilon)>0$ once and for all.
	
	Next, it will be important to have uniform control for distances measured with respect to $g_s$, at least for sufficiently small $s$. This is provided by the following refinement of Claim 1, whose proof is entirely analogous. \\
	
	 \noindent \textbf{Claim 2}: there exists $s_0>0$ with the following property: If $s\in [0,s_0]$, then for every minimising geodesic $\gamma_s$ in $(M^2,g_s)$ with $\gamma_s(0)\in \partial M$ and $L(\gamma_s,g_s)\geq 3\varepsilon/2$, there exists $t\in (0,3\varepsilon/2]$ such that $d_g(\gamma_s(t),\partial M)\geq \delta$.  \\
	
	From now on, we assume the parameter $s$ belongs to the interval $[0,s_0]$ determined by Claim 2.	
	
	The next claim gives an estimate on the length of the piece of a minimising geodesic, joining sufficiently far apart boundary points, outside the collar neighbourhood $V_\delta$. As expected, this piece is fairly long. \\
	
\noindent \textbf{Claim 3}: Let $p$, $q\in \partial M$ be such that $d_g(p,q)\geq \frac{4}{5}\mathcal{S}(\partial M,g)$. If $\gamma$ is a minimising geodesic in $(M^2,g)$ joining $p$ and $q$, then $L(\gamma\setminus V_\delta,g)\geq 3\mathcal{S}(\partial M,g)/5$. \\
	
	In fact, let $\gamma:[0,\ell]\rightarrow M$ be a minimising geodesic of $(M^2,g)$ joining the boundary points $p$ and $q$ in $(M^2,g)$. Since $\ell=L(\gamma)\geq d_g(p,q)> 2\varepsilon$, we may apply Claim $1$ at both extremities and conclude that there exists $t_1\in (0,\varepsilon]$ and $t_2\in [\ell-\varepsilon,\ell)$ such that $d(\gamma(t_i),\partial M)\geq \delta$. By the strict convexity of $\partial M$, a sufficiently small collar neighbourhood of $\partial M$ is foliated by strictly convex curves of $(M^2,g)$. Thus, by the maximum principle, 
	\begin{equation*}
		\gamma([t_1,t_2])\subset \{x\in M\,:\,d(x,\partial M)\geq \delta\}.
	\end{equation*} 
Therefore
\begin{equation*}
	L(\gamma\setminus V_\delta, g) \geq t_2-t_1 \geq \ell-\varepsilon - \varepsilon = L(\gamma) - 2\varepsilon=d_g(p,q)-2\varepsilon.
\end{equation*}
The claim now follows immediately from the definition of $\varepsilon$ and the lower bound on $d_g(p,q)$. \\

	Repeating the proof above using Claim 2 for sufficiently long minimising geodesics of $(M^2,g_s)$ joining boundary points, we deduce the following similar estimate: \\
	
\noindent \textbf{Claim 4}: Let $p$, $q\in \partial M$ be such that $d_{g_s}(p,q)> 3\varepsilon$. If $\gamma_s$ is a minimising geodesic in $(M^2,g_s)$ joining $p$ and $q$, then $L(\gamma_s\setminus V_\delta,g_s)\geq d_{g_s}(p,q)-3\varepsilon$.	\\
		
	We are now ready to make the final estimate. \\

\noindent \textbf{Claim 5}: every sweepout $\{p_t,q_t\}$, $t\in [0,1]$, of $\partial M$ satisfies  
\begin{equation*}
	\max_{t\in[0,1]} d_{g_s}(p_t,q_t)\geq \left(1+\frac{s}{2(1+s)}\right) \mathcal{S}(\partial M,g).
\end{equation*}

	In fact, by definition of $\mathcal{S}(\partial M,g)$, any sweepout of $\partial M$ must contain a pair $\{p,q\}\subset \partial M$ such that $d_g(p,q)\geq 4\mathcal{S}(\partial M,g)/5$. 
We estimate the distance between such points $p$ and $q$ in $(M^2,g_s)$ as follows. Let $\gamma_s$ be a minimising geodesic in $(M^2,g_s)$ joining $p$ and $q$. By properties $i)$ and $ii)$ of the definition of $g_s$, 
	\begin{equation*}
		L(\gamma_s\cap V_\delta,g_s)\geq L(\gamma_s\cap V_\delta,g) \quad \text{and} \quad L(\gamma_s\setminus V_\delta,g_s)= (1+s)L(\gamma_s\setminus V_\delta,g).
	\end{equation*}
	Hence, 
\begin{align*}
	d_{g_s}(p,q) & = L(\gamma_s,g_s) \\
						  & = L(\gamma_s\cap V_\delta,g_s) + L(\gamma_s\setminus V_\delta,g_s)  \\
						  & \geq L(\gamma_s\cap V_\delta,g) + (1+s)L(\gamma_s\setminus V_\delta,g) \\
						  & = L(\gamma_s,g) + s L(\gamma_s\setminus V_\delta,g) \\
						  & \geq d_g(p,q) + s L(\gamma_s\setminus V_\delta,g).
\end{align*}
	By \eqref{eqauxSL1}, we then have
\begin{equation}\label{eqauxSL2}
	d_{g_s}(p,q) \geq d_g(p,q) + \frac{s}{1+s}L(\gamma_s\setminus V_\delta,g_s).
\end{equation}
	Finally, since $d_{g_s}(p,q)\geq d_g(p,q)\geq 4\mathcal{S}(\partial M,g)/5> 3\varepsilon$, we can use Claim 4 and the definition of $\varepsilon$ so to estimate
\begin{align}\label{eqauxSL3}
			L(\gamma_s\setminus V_\delta,g_s)	  & \geq d_{g_s}(p,q)-3\varepsilon \nonumber \\
						  & \geq d_g(p,q)-3\varepsilon \nonumber \\
						  & \geq \frac{4}{5}\mathcal{S}(\partial M,g)-\frac{3}{10}\mathcal{S}(\partial M,g) \nonumber \\
						  & \geq \frac{1}{2}\mathcal{S}(\partial M,g).
\end{align}
	Combining \eqref{eqauxSL2} and \eqref{eqauxSL3}, we conclude that the inequality
\begin{equation*}
	d_{g_s}(p,q)\geq d_g(p,q) + \frac{s}{2(1+s)}\mathcal{S}(\partial M,g)
\end{equation*}
	holds for any pair $\{p,q\}$ of the sweepout that satisfies $d(p,q)\geq 4\mathcal{S}(M,g)/5$. Since, by definition of $\mathcal{S}(\partial M,g)$, there are such pairs with $d(p,q)$ arbitrarily close to $\mathcal{S}(M,g)$ in any sweepout, the Claim follows.  \\
	
	By the definition of $\mathcal{S}(\partial M,g_s)$, the proposition is now an immediate consequence of Claim 5.
\end{proof}

\section{Involutive symmetry}\label{sectioninvolutive}

	Let $\Omega$ be a totally convex disc with smooth boundary in a complete Riemannian surface $(M^2,g)$. In this section, we make an extra assumption about $\partial \Omega$, namely:
\begin{itemize}
	\item[$(\star\star)$] \textit{every two points of $\partial\Omega$ are joined by only one geodesic}.
\end{itemize}
\indent Our purpose is to characterise such discs when the width and the diameter of the boundary coincide, and when there exists an isometric involution $A:(\Omega,g) \rightarrow (\Omega,g)$ that does not fix any point of $\partial \Omega$. This is Theorem F.

\begin{thm}\label{thmFbis}
 Let $\Omega$ be a totally convex disc  with smooth boundary in a complete Riemannian surface $(M^2,g)$. Assume that every two points of $\partial \Omega$ are joined by a unique geodesic, that $\mathcal{S}(\partial \Omega)=diam(\partial \Omega)$, and that $(\Omega,g)$ admits an isometric involution $A$ with no fixed boundary points.

	Then the involution $A$ has a unique fixed point $x_0\in \Omega$, $\mathcal{S}(\partial \Omega)/2$ is not bigger than the injectivity radius of $(M^2,g)$ at $x_0$, and $ \Omega$ is the geodesic ball of $(M^2,g)$ with center at $x_0$ and diameter $\mathcal{S}(\partial \Omega)=diam(\partial \Omega)$.
\end{thm}

\begin{proof} By Lemma \ref{lemcontinuousinvolution}, the restriction of the isometric involution $A : \Omega\rightarrow \Omega$ to $\Gamma$ is a monotone homeomorphism. As a continuous map from the disc to itself, $A$ has fixed points (by Brouwer Theorem), and none of them lies in the boundary (by assumption). Let $x_0\in \Omega$ denote a fixed point of $A$ that is closest to ${\partial \Omega}$. 

	Since $A$ is an involutive isometry, $DA(x_0) : (T_{x_0}\Omega,g_{x_0}) \rightarrow (T_{x_0}\Omega,g_{x_0})$ is a linear isometry such that $DA(x_0)\circ DA(x_0)=Id_{T_{x_0}\Omega}$. In particular, the two eigenvalues of $DA(x_0)$ belong to $\{\pm 1\}$.

	Let $y_0$ be a point of $\partial \Omega$ such that $d(y_0,x_0)$ is the minimum distance between points of $\partial \Omega$ and $x_0$. It follows from the first variation formula \eqref{eqfirstvariation} that any minimising geodesic joining $x_0$ and $y_0$ is orthogonal to $\partial \Omega$ at $y_0$. Thus, there is only one such geodesic from $x_0$ to $y_0$. Call it $\gamma_0 : [0,a] \rightarrow \Omega$, where $a=d(y_0,x_0)$.
	
	Consider the piecewise smooth curve $C$ made of the two arcs $\gamma_0$ and $A(\gamma_0)$.  The extremities of $C$ are the boundary points $y_0$ and $A(y_0)$. This curve is moreover embedded, otherwise the two minimising geodesics issuing from $x_0$ that constitute it would intersect at some point $x\neq x_0$ in the interior of $\Omega$.
	
	Hence, the curve $C$ divides the disc $\Omega$ into two connected components (by Jordan Theorem). Since $A$ restricted to the boundary permutes the arcs determined by $y_0$ and $A(y_0)$, $A$ permutes the two components of $\Omega\setminus C$. Thus, no point in any of these components is fixed by $A$. Even more, $DA(x_0)$ cannot fix any non-zero vector of $T_{x_0} \Omega$ except possibly $\gamma_0'(0)$ and its multiples. However, if it does so, $\gamma_0$ is equal to $A\circ\gamma_0$, which is absurd because their end-points are $y_0\neq A(y_0)$. 
	
	This has two consequences. First, no other point of $\Omega$ besides $x_0$ is fixed by the involution $A$. Second, the number $1$ is not an eigenvalue of $DA(x_0)$. Therefore
	\begin{equation} \label{eqaux1}
		DA(x_0)=-Id_{T_{x_0}M}.
	\end{equation}
	
	Using \eqref{eqaux1}, it is immediate to see that the concatenated curve $C$ is smooth at $x_0$. It is also orthogonal to $\partial \Omega$ at $y_0$ and $A(y_0)$. Since there exists, by property $(\star\star)$, only one geodesic joining these two boundary points, we conclude that $C$ is a free boundary minimising geodesic joining $y_0$ and $A(y_0)$. In particular, $\{y_0,A(y_0)\}\subset \partial \Omega$ is a non-trivial critical point of $\mathcal{D}$, which moreover satisfies
\begin{equation} \label{eqaux2}
	\mathcal{S}(\partial \Omega)= d(y_0,A(y_0)) = L(C)= L(\gamma_0)+L(A(\gamma_0)) = 2d(y_0,x_0).
\end{equation}
The first equality in \eqref{eqaux2} is a consequence of the assumption $\mathcal{S}(\partial \Omega)=diam(\partial \Omega)$. (Recall Proposition \ref{propuniqu} and Proposition \ref{propthmCidafracaB}).

	Continuing the argument, let $y$ be any boundary point. Let $\gamma:[0,b]\rightarrow \Omega$ be some minimising geodesic joining $x_0$ and $y$. Arguing as before, it follows from \eqref{eqaux1} that the two minimising arcs $\gamma$ and $A(\gamma)$ are part of an embedded geodesic passing through $x_0$ with extremities $y$ and $A(y)$ in $\partial \Omega$. By property $(\star\star)$, it must be the unique geodesic joining $y$ and $A(y)$. In particular, it is minimising. Hence,
\begin{equation}\label{eqaux3}
	2d(y,x_0) = L(\gamma)+L(A(\gamma))
	 	= d(y,A(y))\leq diam(\partial \Omega).
\end{equation}	
	
	Since by assumption $\mathcal{S}(\partial \Omega)=diam(\partial \Omega)$, and by construction $d(y_0,x_0)\leq d(y,x_0)$ for all $y\in \partial \Omega$, it follows from \eqref{eqaux2} and \eqref{eqaux3} that
\begin{equation*}
	d(y,x_0)=\mathcal{S}(\partial \Omega)/2 \quad \text{for every} \quad y\in \partial \Omega.
\end{equation*}
	
	Now that we know that the distance between $x_0$ and boundary points is constant, using the first variation formula \eqref{eqfirstvariation} it is straightforward to check that a minimising geodesic joining $x_0$ and any $y\in \partial \Omega$ is orthogonal to $\partial \Omega$ at $y$, and is, therefore, unique.	
	
	Next, we argue that $\Omega$ is a metric and geodesic ball in $(M^2,g)$ of radius $r:=\mathcal{S}(\partial \Omega)/2$ and center $x_0$. Let $\exp$ denote the exponential map of $(M^2,g)$. Let $N$ be the inward-pointing unit normal vector field along $\partial \Omega$. We have proven above that, for every $y\in \partial \Omega$, the curve $t\in [0,r]\mapsto exp_y(tN(y))\in \Omega$ is the unique geodesic joining $y$ to $x_0$, which is therefore minimising. In particular, if $y\neq y'$ in $\partial \Omega$, then these curves meet only at $x_0$.
	
	As a consequence of this, the map 
	\begin{equation*}
		e: (y,t)\in \partial \Omega\times [0,r]\mapsto exp_y(tN(y))\in \Omega
	\end{equation*}
	descends to an injective, continuous map on the disc $\partial \Omega \times[0,r]/\sim$. (Here, $(y,t)\sim (y',t') \Leftrightarrow t=t'=r$). This continuous map between discs clearly restricts to a homemomorphism between the respective boundaries. By a standard topological argument, we conclude that $e$ is surjective as well.

	After proving all this, it is straightforward to check that $\exp_{x_0}: \overline{B_{r}(0)}\subset T_{x_0}M \rightarrow \Omega$ is a bijection, that all points of $\Omega$ lie within distance $r$ from $x_0$, and that all points of $\partial \Omega$ lie at distance $r$ from $x_0$. In other words, $\Omega$ is, in fact, a geodesic ball of $(M^2,g)$ centred at $x_0$ of radius $r=\mathcal{S}(\partial \Omega)/2\leq inj_{x_0}(M,g)$, as we wanted to prove.
\end{proof}

\section{Comparison between min-max quantities} \label{sectioncompare} Let $(M^2,g)$ be a Riemannian disc with non-negative curvature and strictly convex boundary. Recall Section \ref{introcomparison} for the relevant definitions of the min-max quantities $\omega(M,g)$ and $w_*$ associated to $(M^2,g)$. The aim of this section is to briefly describe some possible comparisons between these two quantities and $\mathcal{S}(\partial M)$.  

	We begin by explaining the example depicted in Figure 1 in Section \ref{introcomparison} in details. Let $\rho : [0, +\infty) \rightarrow [0, +\infty)$ be a smooth function such that $\rho(0)=0$, $\rho^{\prime}_+(0) = +\infty$, $\rho^{\prime}>0$ and $\rho^{\prime\prime} \leq 0$. Assume that $\rho(x) = mx+n$, for every $x\geq x_0$, where $m,n>0$. Let $0< x_0 < x_1$. Consider the Riemannian disc $(M^2,g)$ obtained by the rotation of the graph of the function $y = \rho(x)$, $x \in [0, x_1]$, around the $x$-axis. (See Figure \ref{Figura-loop} in Section \ref{introcomparison}). 

	The Gaussian curvature of $(M^2,g)$ is non-negative, and $\partial M$ has strictly positive geodesic curvature. Thus, we can regard it as totally convex disc inside some complete Riemannian surface. (\textit{Cf.} Example \ref{examconvex} in Section \ref{sectionexamples}).
	
	As it is always the case, $\omega(M,g) < L(\partial M)$. (This inequality has been proven in \cite{DonMon}, Section 6). On the other hand, all free boundary geodesics necessarily pass through the tip of $M$. Hence, if we now choose $m$ small and $x_1$ large enough, we guarantee that any free boundary geodesic is longer than $\partial M$. The curve $\beta$ in Figure \ref{Figura-loop} represents one free boundary geodesic, whose length coincides with $w_{\ast}$. 
	
	The main theorem of \cite{DonMon} implies that $\omega(M,g)$ is realised by the length of a geodesic loop $\alpha$ based at a boundary point, also depicted in Figure \ref{Figura-loop}. By the rotational symmetry, $S(\partial M) = diam(\partial M)$. (\textit{Cf.} Example \ref{examrotational} in Section \ref{sectionexamples}). Hence, the endpoints of the two extremal minimising geodesics $\gamma_{-}$ and $\gamma_{+}$ in Figure \ref{Figura-loop} determine a pair of points in $\partial M$ which is a critical point of $\mathcal{D}$ realising $S(\partial M)$. Notice that the velocity vectors of $\gamma_{\pm}$ point inside the open region between $\partial M$ and $\alpha$, because they cannot be part of the loop $\alpha$, or be part of the strictly convex boundary curve, or cross $\alpha$ by virtue of their minimisation property. Thus, $\partial M$ and $\alpha$ act as barriers while minimimising the length of curves with extremities equal to the extremities of $\gamma_{\pm}$.
	
	Therefore the inequalities (\ref{equation-widths-different}) hold in this case, that is,
	\begin{equation*}
    S(\partial M) < \frac{L(\partial M)}{2} < \omega(M, g) < L(\partial M) < w_{\ast}.
\end{equation*}

	We finish this section with the proof of a general comparison, of independent interest.
\begin{thm}\label{thmcompare}
	Let $(M^2,g)$ be a Riemannian disc with non-negative Gaussian curvature and strictly convex boundary. Then
	\begin{equation*}
		S(\partial M)\leq \omega(M, g) \leq w_{\ast}.
	\end{equation*}
	Moreover, if $S(\partial M)= \omega(M, g)$, then $S(\partial M)=w_{\ast}$ and the three numbers are equal to the length of a free boundary minimising geodesic.
\end{thm}
\begin{proof}
Both inequalities can be obtained as consequences of the existence of the sweepouts constructed in Sections 3 and 4 of \cite{DonMon} using the curve shortening process of Birkhoff adapted to the free boundary setting. Indeed, we know that $w_{\ast} = L(\beta)$ for some free boundary geodesic $\beta$, by \cite{Zho}. Proposition \ref{propstarcritical1} implies that $M$ can be swept out by a path of curves $\{c_t\}$ such that $L(c_t) \leq L(\beta)$. These curves have extremities in $\partial M$ and form a sweepout of $M$ that is admissible for $\omega(M,g)$. Therefore, $\omega(M,g)\leq w_{\ast}$.

	If $\omega(M,g)$ is realised as the length of a free boundary geodesic, the same argument produces a path $\{c_t\}$ which induces a sweepout $\{\partial c_t\}$ of $\partial M$ in the sense of Definition \ref{defiweepout}. In this case, we conclude that $S(\partial M)\leq \omega(M,g)$. 

	Alternatively, according to \cite{DonMon}, $\omega(M,g)=L(\alpha)$ is realised as the length of a geodesic loop similarly to the configuration of Figure \ref{Figura-loop}. Let then $R$ be the region bounded by $\partial M \cup \alpha$. We can use Lemma 3.1 of \cite{DonMon} to sweep $R$ out by curves $\{c_t\}$ of lengths bounded by $L(\alpha)$. This is achieved by an application of the adapted Birkhoff's shortening process starting with the loop $\alpha$. As before, $\{\partial c_t\}$ defines a sweepout of $\partial M$ that is admissible to estimate $\mathcal{S}(\partial M)$, which implies that $S(\partial M)\leq \omega(M,g)$.

Notice that the curves $c_t$ with length close to $L(\alpha)$ are not minimising. Indeed, $c_0 = \alpha$ and this curve has the same initial and final point. In particular, $\mathcal{D}(\partial \alpha) = 0$. Similarly, $\mathcal{D}(\partial c_t)$ is close to zero for small values of $t>0$. Therefore, if $\omega(M,g)$ is realised as the length of a loop, then we actually have $S(\partial M,g) < \omega(M,g)$.

To finish the proof, suppose now that $S(\partial M) = \omega(M,g)$. By the previous analysis, these two numbers are realised as the length of a free boundary geodesic $\beta$. It follows that $w_{\ast}$ must coincide with the first two invariants. Recall that the sweepout $\{c_t\}$ of $M$ obtained before satisfies also $L(c_t) < L(c_{1/2}) = L(\beta)$ for all $t\neq 1/2$. If $\beta$ is not minimising, then $\mathcal{D}(\partial \beta) < L(\beta)$. This implies that the maximum of $\mathcal{D}(\partial c_t)$ is strictly smaller than $L(\beta) = \omega(M,g) = S(\partial M)$, a contradiction. In conclusion, $\beta$ is a free boundary minimising geodesic, and $\partial \beta$ is a critical point of $\mathcal{D}$.
\end{proof}

\section{Examples} \label{sectionexamples}

\begin{exam} \label{examhemisphere}
	The equator of the two-dimensional Euclidean sphere of radius one is a smoothly embedded curve of length $2\pi$. Any critical point of $\mathcal{D}$ consists of a pair of antipodal boundary points, which are joined by minimising geodesics of length $\pi$ that form a set parametrised by a circle. Thinking of the equator as the boundary of the southern hemisphere, it is clear that any pair of antipodal points is joined by a unique free boundary minimising geodesic, while any two geodesics that lie in different components of the hemisphere determined by the free boundary geodesic are simultaneously stationary. The width of the equator is equal to half of its length. This closed embedded geodesic has Morse index one. 
\end{exam}

\begin{exam}\label{examconvex}
	Let $(M^2,g)$ be a Riemannian surface with compact boundary that is strictly convex in the sense that the geodesic curvature vector of $\partial M$ points strictly inside $M$. Then $(M^2,g)$ can be isometrically embedded into a larger complete Riemannian manifold in such way that the complement of $M$ is foliated by circles with geodesic curvature vector pointing towards $M$. (Proof: move $M$ inside by the gradient flow of the distance to the boundary, pull back the metric, and blow up the collar region into ends that have the property above, by a suitable conformal factor). By a curvature comparison argument for curves that lie on one side of the other and intersect tangentially, the latter condition implies that $M$ contains all geodesics of this larger manifold with extremities in $M$. In other words, $(M^2,g)$ is a totally convex Riemannian surface with smooth compact boundary in some complete Riemannian surface. In this case, the width of each boundary component of $\partial M$ is a geometric invariant of $(M^2,g)$.
\end{exam}

\begin{exam}\label{examrotational}
Let $(M^2,g)$ be a Riemannian disc with strictly convex boundary as in Example \ref{examconvex}. Assume, in addition, that $(M^2,g)$ is rotationally symmetric. We claim that all non-trivial critical points of $\mathcal{D}$ are of the form $\{x, A(x)\}\subset \partial M$, where $A: M \rightarrow M$ represents the rotation by $180$ degrees. In particular, $d(x,A(x)) = diam(\partial M)$ = $\mathcal{S}(\partial M)$ for all $x\in \partial M$.

Indeed, let $\{x,y\}\subset \partial M$ be an arbitrary critical point of $\mathcal{D}$. If $R_{\theta} : M\rightarrow M$ is a rotation by $\theta\in[0,2\pi)$, then $\{R_{\theta}(x), R_{\theta}(y)\}\subset \partial M$ is also a critical point of $\mathcal{D}$, because $R_\theta$ is an isometry. The uniqueness property (Proposition \ref{propuniqu}) then implies that if $\theta_0$ is such that $R_{\theta_0}(x)=y$, then $R_{\theta_0}(y)=x$. Therefore $R_{2\theta_0}(x)=x$, \textit{i.e.} $y=R_{\theta_0}(x) =A(x)$ is the rotation of $x$ by $180$ degrees. The claim follows.
\end{exam}

\begin{exam} \label{examellipsoids}
	Given $a>0$, let
	\begin{equation*}
		\mathcal{E}_a=\Big\{(x,y,z)\in \mathbb{R}^3\,|\, x^2+y^2+\frac{z^2}{a^2} = 1 \Big\}
	\end{equation*} 
	be an ellipsoid of revolution around the $z$ axis. Given $\delta>0$, the disc $M_{a,\delta}=\mathcal{E}_{a}\cap\{z\geq \delta\}$ has non-negative Gaussian curvature and strictly convex boundary. According to Example \ref{examrotational}, the critical points of $\mathcal{D}$ are of the form $\{(x,y,\delta),(-x,-y,\delta)\}\subset \partial M$ and satisfy
	\begin{equation*}	
		d((x,y,\delta),(-x,-y,\delta))=\mathcal{S}(\partial M_{a,\delta})=diam(\partial M_{a,\delta}).
	\end{equation*}
	
	For all $a>0$ and $\delta>0$, the critical points bound an unstable free boundary geodesic that goes through the point $(0,0,a)$.
	
	For $0<a<1$, the critical points bound exactly one minimising geodesic, which is the free boundary geodesic with length $<L(\partial M_{a,\delta})/2$.
	
	For $a>1$, if $\delta>0$ is sufficiently small, then none of the free boundary geodesics are minimising, because they all go through $(0,0,a)$ and therefore have length strictly bigger than half the length of $\partial M_{a,\delta}$. The boundary points $(x,y,\delta)$ and $(-x,-y,\delta)$ bound two minimising geodesics that are simultaneously stationary (in fact, one is the reflected image of each other, by the reflection that fixes the free boundary geodesic joining $(x,y,\delta)$ and $(-x,-y,\delta)$).
\end{exam}

\begin{exam}\label{examhalfdumbbell}
	
Let $\Omega$ be the plane region bounded by an ellipse determined by the equation $x^2/a^2+y^2/b^2=1$, where $a\geq b>0$ will be chosen. Let $R_1, R_2\subset \Omega$ be the regions represented in Figure \ref{sugarloaf}. Assume that $\Omega\setminus(R_1\cup R_2)$ is the union of a small collar neighbourhood of $\partial \Omega$ with a small neighbourhood of the segment representing the minor axis of the ellipse.

\begin{figure}[htp]
    \centering
    \includegraphics[width=12cm]{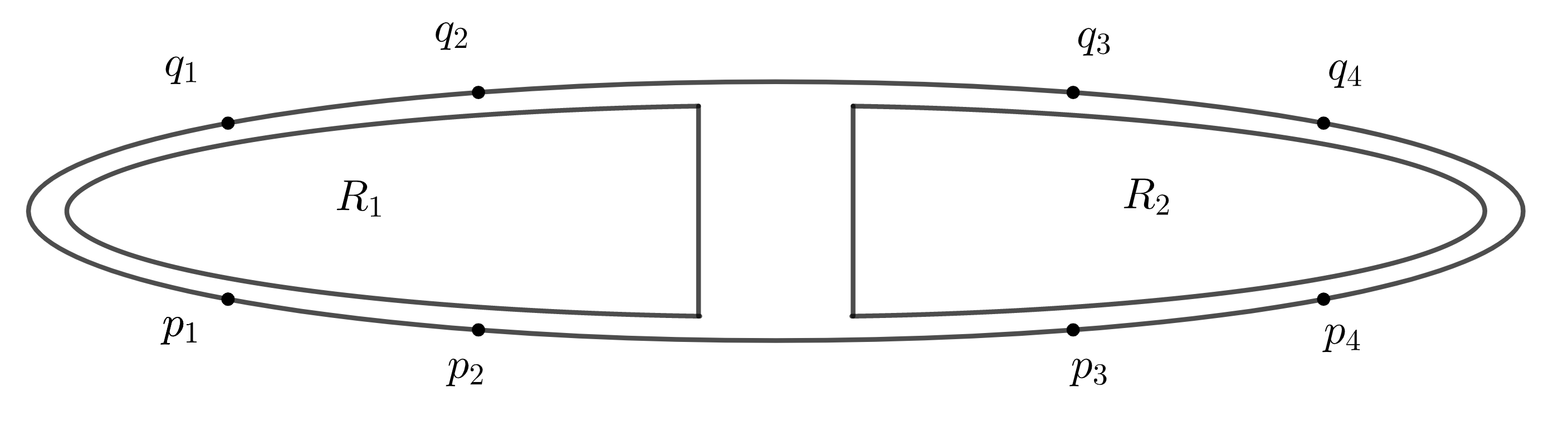}
    \caption{}
    \label{sugarloaf}
\end{figure}

Let $g$ be a Riemannian metric on $\Omega$ that coincides with the Euclidean flat metric outside a small neighbourhood of $R_1\cup R_2$ contained in $\Omega$, and is very large compared to the Euclidean metric in $R_1\cup R_2$ (in the sense that the size of tangent vectors at points in $R_1\cup R_2$ measured according to $g$ are much larger than their Euclidean sizes). In particular, the metric is Euclidean near $\partial \Omega$. Clearly, $\Omega$ is a totally convex disc with smooth boundary in $(\mathbb{R}^2,g)$, where $g$ extends as the Euclidean outside $\Omega$. We claim that $(\Omega,g)$ contains a minimising geodesic $\gamma$ that is free boundary stable, and that
\begin{equation*}
	L(\gamma) < 2b< \mathcal{S}(\partial \Omega),
\end{equation*}
as soon as $b$ and $a$ are chosen appropriately.

In order to see this, let $p_i$ and $q_i$ be points in $\partial \Omega$ as indicated in Figure \ref{sugarloaf}. Let $A$ be the closed arc of the ellipse from $q_3$ to $q_2$ in the counterclockwise direction, and $B$ be the arc from $p_2$ to $p_3$. Let $q\in A$ and $p\in B$ be a pair of points realising the distance between $A$ and $B$ with respect to the metric induced by $g$. If the arcs $A$ and $B$ are large enough compared to $2b$, then $q \in int(A)$, and $p\in int(B)$. Indeed, a path joining $p_2$ or $p_3$ to the arc $A$ either crosses the region $R_1\cup R_2$, where the metric is large, or goes around this region. Making $p_2$ and $p_3$ far enough from the $y$-axis and from the vertices of the ellipse, we conclude that the distance with respect to $g$ from $p_2$ or $p_3$ to $A$ is larger than $2b$. On the other hand, since the metric is flat near the minor axis we have $d(A,B)<2b$. This shows that $p\in int(B)$. Similarly we have $q\in int(A)$. In particular, $\{p,q\}$ is a local minimum of $\mathcal{D}$, which implies that there exists a free boundary minimising geodesic $\gamma$ from $p$ to $q$ which is free boundary stable (Proposition \ref{proplocalminima}) and $L(\gamma)<2b$.

Let $C$ be the arc of the ellipse from $q_1$ to $p_1$ in the counterclockwise direction, and $D$ be the arc from $p_4$ to $q_4$. Assume that $p_1$ is far from $p_2$, $q_1$ is far from $q_2$, and both $p_1$ and $q_1$ are far from the tip of the ellipse compared to the value $2b$. Therefore, $d(p_1,q_1)$ is also large compared to $2b$. Moreover, if $y\in \partial \Omega$ and $d(p_1,y)< 3b$, then $y$ belongs to the open arc of the ellipse from $q_1$ to $p_2$. The points $p_3$, $p_4$, $q_3$, and $q_4$ are centrally symmetric to $p_2$, $p_1$, $q_2$, and $q_1$, respectively. Therefore they have analogous properties.

Finally, we can argue that the width of $\partial \Omega$ in $(\Omega,g)$ is strictly bigger than $2b$. Fix a sweepout $\{p_t, q_t\}$ of $\partial\Omega$. Arguing as in the proof of Proposition \ref{propthmCvolta}, we conclude that every sweepout contains a pair $\{x,-x\}$ of centrally symmetric points of $\partial\Omega$. If $x$ belongs to one of the arcs $C$ or $D$, then $-x$ belongs to the other, and $d(x, -x) \geq 3b$. If $x$ belongs to the arc from $p_1$ to $p_4$, then $-x$ belongs to the arc from $q_4$ to $q_1$. Vary the pair $\{x,-x\}$ continuously by pairs of the sweepout until the first time one has a pair $\{p_{t_0},q_{t_0}\}$ with a point in $C\cup D$. This happens because otherwise $p_t$ and $q_t$ would be trapped in the disjoint arcs from $q_4$ to $q_1$, and from $p_1$ to $p_4$. Without loss of generality, suppose $p_{t_0} = p_1$. Our choice implies that $q_{t_0}$ belongs to the arc from $q_4$ to $q_1$. The property obtained in the preceding paragraph implies that $d(p_{t_0},q_{t_0})\geq 3b$. Since the sweepout is arbitrary, we conclude that $S(\partial\Omega, g) > 2b$.
\end{exam}

\begin{exam}\label{exammodifiedellipsoid}
	Consider again the discs $M_{a,\delta}$ inside the ellipsoid $\mathcal{E}_{a}$ from Example \ref{examellipsoids}, where $a>1$ and $\delta>0$ is sufficiently small. Pick $\varepsilon>0$ and denote by $B_\varepsilon$ the open metric ball around $(0,0,a)$ of radius $\varepsilon>0$. If $\varepsilon>0$ is sufficiently small depending only on $a$ and $\delta$, no curve in $M_{a,\delta}$ that joins points $(x,y,\delta)$ and $(-x,-y,\delta)$ and intersects $B_\varepsilon$ is minimising. Thus, no modification of the metric with compact support in $B_{\varepsilon}$ changes the value of $\mathcal{D}$ on pairs of points in $\partial M_{a,\delta}$. In fact, one may even attach handles and cross caps inside this ball, without changing the fact that the width and the diameter of $\partial M_{a,\delta}$, as the boundary of these compact totally convex Riemannian surfaces with different topologies, are equal.
	
	In particular, we may even arrange this modification of the metric supported on $B_\varepsilon$ in such way that the map $A:x\in \Omega_{a,\delta }\mapsto -x\in \Omega_{a,\delta}$ is still an isometry without fixed points in $\partial \Omega_{a,\delta}$, while $\Omega_{a,\delta}$ is not rotationally symmetric. 
\end{exam}

\begin{exam}\label{examthmF}
	There are many other interesting examples of smoothly embedded discs in Riemannian manifolds whose boundaries satisfy one of the equivalent properties described in Theorem C (Theorem \ref{thmCbis}). For instance, let $x_0$ be a point in a complete Riemannian surface $(M^2,g)$, and assume $r>0$ is smaller than the convexity radius of $(M^2,g)$ at $x_0$. Then the geodesic ball $B_r(x_0)$ is diffeomorphic to a disc and $\partial B_r(x_0)$ is a smoothly embedded, strictly convex circle, and moreover the minimising geodesics joining points of $B_r(x_0)$ lie in $B_{r}(x_0)$ and are unique. (This is stronger than $(\star\star)$). Given a unit vector $v\in T_{x_0}M$, the geodesic $t\in [-r,r]\mapsto \exp_{x_0}(tv)\in M$ is the unique minimising geodesic joining the points $\exp_{x_0}(\pm r v)$. It has length $2r$. The obvious continuous map $\exp_{x_0}(-rv) \in \partial B_r \mapsto exp_{x_0}(rv)\in \partial B_r$ is therefore a smooth map such that $d(\exp_{x_0}(-rv),\exp_{x_0}(rv))=2r$. Since by the triangle inequality we have $diam(\partial B_r)=2r$, it follows from Proposition \ref{propthmCvolta} that $\mathcal{S}(\partial B_r)=diam(\partial B_r)$.
	
	This class of examples also shows that the characterisation in Theorem F (Theorem \ref{thmFbis}) is sharp, because if a sufficiently small ball $B$ has an isometric involution fixing the center of $B$, then it has properties $(\star\star)$ and $\mathcal{S}(\partial B)=diam(B)$. Notice that such ball is not necessarily rotationally symmetric.

\end{exam}

\begin{exam}\label{examellipsoids3}
	Consider the ellipsoid $\mathcal{E}_a$ of Example \ref{examellipsoids}. When $a>1$, the vertical geodesic $\Gamma=\mathcal{E}_a\cap\{x=0\}$ satisfies $\mathcal{S}(\Gamma)<diam(\Gamma)=L(\Gamma)/2$. This can be seen from the estimates for the horizontal sweepouts determined by the intersection of $\Gamma$ with horizontal planes, and from Lemma \ref{lemdiamlength}.
\end{exam}

\begin{exam}\label{examfilling}
	When $a>1$, the limit $\delta \rightarrow 0$ of the examples described in Example \ref{exammodifiedellipsoid} are examples of fillings of the unit circle $\{(x,y)\in \mathbb{R}^2\,|\, x^2+y^2=1\}$. This is a particular example of a much more general construction that yields many examples of smoothly embedded circles $\Gamma$ with $\mathcal{S}(\Gamma)=L(\Gamma)/2$ in Riemannian spheres. (\textit{Cf.} \cite{Gro1}).
	
	Let $g$ be a Riemannian metric on the real projective plane $\mathbb{RP}^2$. Let $\gamma$ be the simple closed geodesic that realises the least length among homotopically non-trivial loops of $(\mathbb{RP}^2,g)$. The metric $g$ lifts to a Riemannian metric on the sphere $S^2$, and $\gamma$ lifts to a simple closed geodesic $\Gamma$ in $(S^2,g)$. We claim that $\mathcal{S}(\Gamma)=L(\Gamma)/2$. In fact, pick any point $x\in \Gamma$, and consider the lift of $\gamma$ to $(S^2,g)$ based at $x$. Since $\gamma$ is homotopically non-trivial, the lift is an open arc of length $L(\gamma)=L(\Gamma)/2$ that ends at a point $\phi(x)$. Any other curve $c$ joining $x$ and $\phi(x)$ in $(S^2,g)$ projects down to a homotopically non-trivial loop, whose length is therefore at least $L(\gamma)=L(\Gamma)/2$. Thus, $d(x,\phi(x))=L(\Gamma)/2$. In particular, $diam(\Gamma)=L(\Gamma)/2$.  Since the map $x\in \Gamma \mapsto \phi(x)\in \Gamma$ is continuous (it is just the restriction to $\Gamma$ of the deck transformation of $S^2$ that produces $\mathbb{RP}^2$ as a quotient), the claim now follows by Theorem C (Theorem \ref{thmCbis}).
\end{exam}

\begin{exam}\label{examcylhem}
	Consider the cylinder $S^1\times \mathbb{R}$ endowed with the flat product metric. Since it is a complete surface foliated by geodesics $\Gamma_t=S^1\times\{t\}$, $t\in \mathbb{R}$, it is straightforward to check that the distance between points in $\Gamma_t$ is realised by the shortest arc of $\Gamma_t$ determined by these points. By Theorem C (Theorem \ref{thmCbis}), the curves $\Gamma_t$ satisfies $S(\Gamma_t)=L(\Gamma_t)/2$. Notice that, as closed geodesics, $\Gamma_t$ have Morse index zero. 
	
	An example of smoothly embedded circle with the same properties inside a compact surface can be obtained by capping the cylinder sufficiently far from $\Gamma_0$ on both sides, so that geodesics with extremities in $\Gamma_0$ that leave the cylindrical region have length bigger than $L(\Gamma_0)/2$.
\end{exam}

\begin{exam} \label{examellipse}
	For every $a\neq 1$, the ellipse 
	\begin{equation*}
		\Omega_{a} = \Big\{ (x,y)\in \mathbb{R}^2\,|\, x^2+\frac{y^2}{a^2} = 1  \Big\}
	\end{equation*}	
	bounds a disc with strictly convex boundary that contains only two non-trivial critical points of $\mathcal{D}$, namely, $\{(-1,0),(1,0)\}$ and $\{(0,-a),(0,a)\}$.  
\end{exam}

\begin{exam}
	Going beyond space forms, different notions of convex subsets of constant width in Riemannian manifolds have been proposed. For instance, generalizations of some classical results for convex bodies in the Euclidean space are obtained in \cite{Dek}, under much more restrictive convexity conditions than the ones ever assumed in this paper.
	
	The works of Robertson \cite{Rob} and Bolton \cite{Bol} proposed a different generalisation, which does not involve any convexity assumption and has a more topological flavour. They call a hypersurface $N$ of a complete Riemannian manifold a \textit{transnormal hypersuface} when every geodesic that intersects $N$ orthogonally at one point intersects $N$ orthogonally at every other point of intersection. These hypersurfaces enjoy interesting topological and geometric properties (see \cite{Rob}, \cite{Bol}, \cite{Nis}, \cite{Nis2}, \cite{Weg1} and \cite{Weg2} for a non-exhaustive list of works on this subject).
	
	If a transnormal circle $\Gamma$ in a Riemannian surface is the boundary of a compact totally convex region $(\Omega,g)$, then for every point $x$ there exists a unique $\phi(x)\neq x$ in $\Gamma$ such that $\{x,\phi(x)\}$ bounds a unique free boundary geodesic. Indeed, every such $\Omega$ contains a free boundary geodesic $\gamma$. Let $\{p,q\} = \partial \gamma$. Since $\gamma$ crosses $\partial \Omega$ transversally at $p$ and $q$, if $\tilde{p} \in \partial \Omega$ is sufficiently close to $p$, then the geodesic $\tilde{\gamma}$ starting at $\tilde{p}$ normal to $\partial \Omega$ will cross the boundary curve a second time at a point $\tilde{q}$ near $q$. The transnormality of $\partial \Omega$ implies that the portion of $\tilde{\gamma}$ connecting $\tilde{p}$ and $\tilde{q}$ is a free boundary geodesic. The compactness of the set of proper free boundary geodesics finishes the argument.

	Moreover, by the first variation formula \eqref{eqfirstvariation}, all the free boundary geodesics connecting $x$ and $\phi(x)$ have the same length. Also, $\phi(\phi(x))=x$. The map $x\in \Gamma\mapsto \phi(x)\in \Gamma$ is therefore a monotone homeomorphism (Lemma \ref{lemcontinuousinvolution}). 

	It may happen that these free boundary geodesics are not minimising geodesics, so it is not clear that $\{x,\phi(x)\}$ are critical points of $\mathcal{D}$. However, if moreover $(\Omega,g)$ has property $(\star\star)$, then it is clear that these free boundary geodesics are minimising, so that $\{x,\phi(x)\}$, $x\in \Gamma$, are in fact critical points of $\mathcal{D}$. By Theorem C (Theorem \ref{thmCbis}), it follows that $\mathcal{S}(\partial \Omega)=diam(\partial \Omega)$ in this case.
	
	The class of examples discussed in Example \ref{exammodifiedellipsoid} contains examples of discs with strictly convex boundary with $\mathcal{S}(\partial \Omega)=diam(\partial \Omega)$ that are not transnormal. In what follows, we describe three examples of curves that are not transnormal, and which nevertheless have the same width and diameter. 
	
	Let $\mathcal{E} = \mathcal{E}(a,b,c)$ be an ellipsoid $x^2/a^2+y^2/b^2+z^2/c^2=1$, with $a> b > c$, all close to $1$. This surface has exactly simple three simple closed geodesics, the curve $\Gamma = \mathcal{E}\cap\{x=0\}$ being the shortest one. We claim that $S(\Gamma) = diam(\Gamma) = L(\Gamma)/2$, but $\Gamma$ is not transnormal. In fact, notice that $\mathcal{E}$ is the Riemannian lift of a Riemannian $\mathbb{RP}^2$, and that $\Gamma$ is the lift of the shortest homotopically non-trivial curve in that $\mathbb{RP}^2$. As in Example \ref{examfilling}, the first part of the claim follows. On the other hand, since the two symmetric halves of the orthogonal geodesics $\mathcal{E}\cap\{y= 0\}$ and $\mathcal{E}\cap\{z= 0\}$ to $\Gamma$ have different lengths, $\Gamma$ is not transnormal.

	Next we explain how to obtain an example from those of Example \ref{exammodifiedellipsoid}. Recall that these examples have a rotationally symmetric part, and a region where the metric is modified. In the present context, it is convenient to modify the metric in such a way that the tip $B_{\varepsilon}$ is the graph of a positive function $z=f(x,y)$ over a disc, with two points of maximum at $(\alpha, 0)$ and $(-\alpha, 0)$ and one saddle point at $(0,0)$. Assume that the disc $\Omega$ is invariant under the reflections with respect to the $xz$-plane and the $yz$-plane. This surface resembles a mountain with rotationally symmetric base, and two peaks. The symmetries imply that the curves $\gamma_1 = \Omega\cap\{y=0\}$ and $\gamma_{2} = \Omega\cap \{x=0\}$ are free boundary geodesics. The curve $\gamma_1$ goes up the two peaks, while $\gamma_2$ has maximum height at saddle point. If $L(\gamma_1)\neq L(\gamma_2)$, then $\partial \Omega$ cannot be transnormal.

	Finally, we present an example in higher codimension. Let $\Gamma$ be a smoothly embedded circle in the Euclidean space that is contained in the unit sphere and is centrally symmetric. Since $d(x,y) = ||x-y|| \leq 2$ and $d(x, -x) = 2$ for every $x$, $y \in \Gamma$, we have $diam(\Gamma) = 2$. Moreover, the map $\phi : \Gamma \rightarrow \Gamma$ defined by $\phi(x) = -x$ is continuous and satisfies $d(x, \phi(x)) = 2$. Thus, Theorem C (Theorem \ref{thmCbis}) implies that $S(\Gamma) = diam(\Gamma)$. However, not all such curves $\Gamma$ are transnormal in $(\mathbb{R}^3,can)$. 	
\begin{figure}[htp]
    \centering
    \subfloat{{\includegraphics[width=4.5cm]{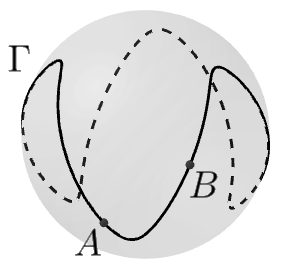} }}%
    \qquad
    \subfloat{{\includegraphics[width=4.5cm]{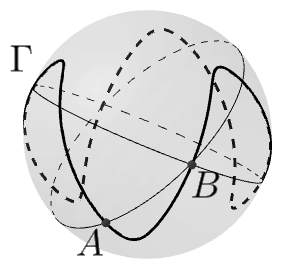} }}%
    \caption{Curve $\Gamma$, two points $A, B\in \Gamma$ (left), and great circles normal to $\Gamma$ at the points $A$ and $B$.}%
    \label{not-transnormal-curve}%
\end{figure}	
	
	We exemplify one of such curves in Figure \ref{not-transnormal-curve}. On the left part of the figure, we see the curve $\Gamma$ inside the unit sphere and two points $A$, $B \subset \Gamma$ such that the line segment from $A$ to $B$ is perpendicular at $A$, but not at $B$. On the right part of the figure, this property can be better visualised. The depicted great circles in the unit sphere which are normal to $\Gamma$ at $A$ and $B$ define the planes that contain all segments orthogonal to $A$ and $B$, respectively. These circles cross $\Gamma$ in 6 points, $A$ and $B$ included. However, the line segment connecting $A$ and $B$ is not normal at $B$, since $A$ is not contained in the normal circle of $\Gamma$ at $B$.
\end{exam}

\end{document}